\documentclass[11pt,a4paper]{article}
\usepackage[utf8]{inputenc}
\usepackage[english]{babel}
\usepackage{pst-grad} 
\usepackage{pst-plot} 
\usepackage{pstricks}
\usepackage{amsmath,amssymb,mathrsfs,amsthm}
\usepackage[dvips]{graphicx}
\usepackage{anysize} 
\marginsize{3cm}{2cm}{2cm}{2cm} 
\usepackage{epsfig}
\newcommand{\re}{\text{Re}}
\newcommand{\im}{\text{Im}}
\newcommand{\BB}{\mathbb B}

\newcommand{\RR}{\mathbb R}

\newcommand{\pat}{\partial_t}
\newcommand{\pax}{\partial_x}

\newcommand{\paa}{\partial_\alpha}

\newcommand{\jeps}{\mathcal{J}_\epsilon*}

\usepackage[pdfauthor={Rafael Granero-Belinch\'on},pdftitle={},bookmarks,colorlinks]{hyperref}
\hypersetup{colorlinks=true}
\newcounter{comentcount}
\newcounter{teocount}
\newcounter{propcount}
\newtheorem{prop}[propcount]{Proposition}

\newtheorem{thm}[teocount]{Theorem}  
\newtheorem{defi}{Definition}

\newenvironment{coment}
{\stepcounter{comentcount} {\bf \tt Remark} {\bf\tt\arabic{comentcount}} }{ }
\title{The confined Muskat problem: differences with the deep water regime}
\author{Diego C\'ordoba Gazolaz$^{\mbox{{\footnotesize 1},{\footnotesize 4}}}$, Rafael Granero-Belinch\'on$^{\mbox{{\footnotesize 2},{\footnotesize 4}}}$ and Rafael Orive Illera$^{\mbox{{\footnotesize 3},{\footnotesize 4},{\footnotesize 5}}}$}

\begin{document}

\maketitle 

\footnotetext[1]{Email: \texttt{dcg@icmat.es}}

\footnotetext[2]{Email: \texttt{r.granero@icmat.es}}

\footnotetext[3]{Email: \texttt{rafael.orive@icmat.es}}

\footnotetext[4]{Instituto de Ciencias Matem\'aticas CSIC-UAM-UC3M-UCM, Consejo Superior de Investigaciones Cient\'ificas, C/Nicol\'as Cabrera, 13-15, Campus de Cantoblanco, 28049 - Madrid}

\footnotetext[5]{Departamento de Matem\'aticas, Facultad de Ciencias, Universidad Aut\'onoma de Madrid, Campus de Cantoblanco, 28049 - Madrid}

\vspace{0.3cm}

\begin{abstract}
We study the evolution of the interface given by two incompressible fluids with different densities in the porous strip $\RR\times[-l,l]$. This problem is known as the Muskat problem and is analogous to the two phase Hele-Shaw cell. The main goal of this paper is to compare the qualitative properties between the model when the fluids move without boundaries and the model when the fluids are confined. We find that, in a precise sense, the boundaries decrease the diffusion rate and the system becomes more singular.
\end{abstract}

\vspace{0.3cm}


\textbf{Acknowledgments}: The authors are supported by the Grants MTM2011-26696 and SEV-2011-0087 from Ministerio de Ciencia e Innovaci\'on (MICINN). Diego C\' ordoba was partially supported by StG-203138CDSIF of the ERC. Rafael Granero-Belinch\'on is grateful to A. Castro and F. Gancedo for their helpful comments during the preparation of this work. 

\section{Introduction}\label{sec0}
In this paper we study the evolution of the interface between two different incompressible fluids with the same viscosity in a flat two-dimensional strip. This problem has an interest because it is a model of an aquifer or an oil well, see  \cite{Muskat}. In this phenomena, the velocity of a fluid in a porous medium satisfies Darcy's law 
\begin{equation}
\frac{\mu}{\kappa}v=-\nabla p-g\rho e_2,
\label{eq1} 
\end{equation}
where $\mu$ is the dynamic viscosity, $\kappa$ is the permeability of the medium, $g$ is the acceleration due to gravity, $\rho$ is the density of the fluid, $p$ is the pressure of the fluid and $v$ is the incompressible field of velocities, see \cite{bear, bn}.

Equation \eqref{eq1} has also been considered as a model of the velocity for cells in tumor growth, see for instance \cite{F,P} and references therein.

The motion of a fluid in a two-dimensional porous medium is analogous to the Hele-Shaw cell problem, see \cite{H-S}. In this case the fluid is trapped between two parallel plates. The mean velocity in the cell is described by
\begin{equation*}
\frac{12\mu}{b^2}v=-\nabla p-g\rho e_2,
\end{equation*}
where $b$ is the (small) distance between the plates.

We consider the two-dimensional flat strip $S=\RR\times(-l,l)\subset\RR^2$ with $l>0$. In this strip we have two immiscible and incompressible fluids with the same viscosity and different densities, $\rho^1$ in $S^1(t)$ and $\rho^2$ in $S^2(t)$, where $S^i(t)$ denotes the domain occupied by the $i-$th fluid. The curve $$z(\alpha,t)=\{(z_1(\alpha,t),z_2(\alpha,t)): \: \alpha\in\RR\}$$ is the interface between the fluids. We suppose that the initial interface $f_0(x)$ is a graph and $|f_0(x)|\leq l$ for all $x$. The character of being a graph is preserved at least for a short time (see Section \ref{sec2}). The Rayleigh-Taylor condition is defined as
\begin{equation*}
RT(\alpha,t)=-(\nabla p^2(z(\alpha,t))-\nabla p^1(z(\alpha,t)))\cdot\partial_\alpha^\bot z(\alpha,t).
\end{equation*}
Due to the incompressiblity of the fluids and using that the curve can be parametrized as a graph, the Rayleigh-Taylor condition reduces to the sign of the jump in the density:
\begin{equation*}
RT=g(\rho^2-\rho^1)>0.
\end{equation*}
This condition is satisfied if the denser fluid is below.

\setcounter{figure}{0}
\begin{figure}[t]
		\begin{center}
		\includegraphics[scale=0.6]{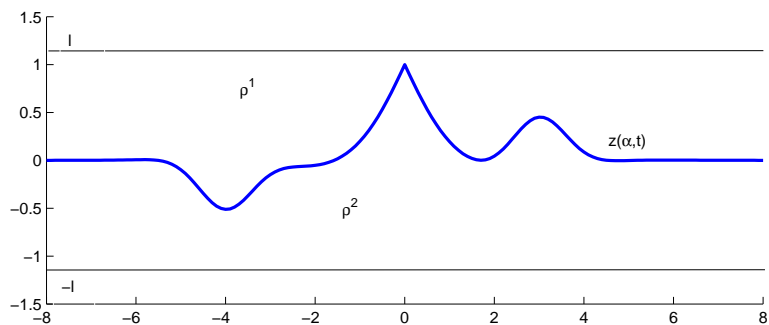} 
		\end{center}
		\caption{Physical situation for an interface $z(\alpha,t)$ in the strip $\RR\times (-l,l)$.}
\label{bandafig}
\end{figure}

We consider the velocity field $v$, the pressure $p$ and the density $\rho$ 
\begin{equation}
\rho(t)=\rho^1\textbf{1}_{S^1\left (t\right )}+\rho^2\textbf{1}_{S^2\left (t\right )},
\label{eq3a}
\end{equation}
in the whole domain $S$. We also consider the conservation of mass equation, so we have a weak solution to the following system of equations
\begin{equation}
\left\{\begin{array}{l}
\displaystyle\frac{\mu}{\kappa}v(x,y,t) = -\nabla p(x,y,t)-g\rho e_2\quad\text{in } S, \: t>0,\label{eq4}\\
\nabla\cdot v(x,y,t) = 0\quad\text{in } S, \: t>0,\\
\pat \rho(x,y,t)+v\cdot\nabla\rho(x,y,t)=0\quad\text{in } S,\: t>0,\\
f(x,0) = f_0(x) \quad\text{in } \RR,\\
v(x,\pm l)\cdot n=0 \quad\text{in } \RR,
\end{array}\right.
\end{equation}
\emph{i.e.} with impermeable boundary conditions for the velocity.. 

We denote by $v^1(x,y,t)$ the velocity field in $S^1(t)$ and by $v^2(x,y,t)$ the velocity field in $S^2(t)$. Because of the incompressibility condition, the normal components of the velocities $v^1,v^2$ are continuous through the interface. Moreover, the interface moves along with the fluids. Therefore, if initially we have an interface which is the graph of a function, we have the following equation for the interface:

\begin{equation}
\pat f(x,t)=(-\pax f(x,t),1)\cdot v^i(x,f(x,t),t)=\sqrt{1+(\pax f(x,t))^2}n\cdot v^i, 
\label{eq5}
\end{equation}
where $n$ denotes the unit normal to the interface.

In each subdomain $S^i(t)$ the fluids satisfy Darcy's law \eqref{eq1},
\begin{equation}
\frac{\mu}{\kappa}v^i(x,y,t)=-\nabla p^i(x,y,t)-g\rho^i(0,1)\quad\text{in } S^i(t),
\label{eq6} 
\end{equation}
and the incompressibility condition
\begin{equation}
\nabla\cdot v^i(x,y,t)=0 \quad\text{in } S^i(t).
\label{eq7}
\end{equation}

We define the following dimensionless parameter (see \cite{bona2008asymptotic} and references therein)
\begin{equation}\label{dimensionless}
\mathcal{A}=\frac{\|f_0\|_{L^\infty}}{l}. 
\end{equation}
This parameter is called the \emph{nonlinearity} (or \emph{amplitude}) parameter and we have $0\leq\mathcal{A}\leq 1$. 

The case $\mathcal{A}=1$ is the case where $f$ reaches the boundaries and we call it the \emph{large amplitude regime}. In \cite{knupfer2010darcy} they consider a two dimensional droplet in vacuum over a plate driven by surface tension. 

The case $\mathcal{A}=0$ is the \emph{deep water regime} for which the equation reduces to
\begin{equation}\label{full}
\pat f=\frac{\rho^2-\rho^1}{2\pi}\text{P.V.}\int_\RR \frac{(\pax f(x)-\pax f(x-\eta))\eta}{\eta^2+(f(x)-f(x-\eta))^2}d\eta.
\end{equation}
It has been shown, for equation \eqref{full}, local existence in Sobolev spaces when the Rayleigh-Taylor condition holds (see \cite{c-g07}), a maximum principle for the $L^\infty$ norm of $f$ and also a maximum principle for $\|\pax f\|_{L^\infty}$ (see \cite{c-g09}). For initial data with $\|\pax f_0\|_{L^\infty}<1$ follows global existence of $W^{1,\infty}$ solution (see \cite{ccgs-10}). For large initial datum there are \emph{turning waves}, \emph{i.e} a blow up for $\|\pax f\|_{L^\infty}$ (see \cite{ccfgl}). For other results see \cite{ambrose2004well, castro2012breakdown, ccfgl, c-c-g10, KK, SCH}.

The equation for the evolution of the interface in our bounded domain, which is deduced in Section \ref{sec1}, is 
\begin{multline}
\pat f(x,t) = \frac{\rho^2-\rho^1}{8l}\text{P.V.}\int_\RR\bigg{[}\left (\partial_xf\left (x\right )-\partial_xf\left (x-\eta\right )\right)\Xi_1(x,\eta,f)\\
+ (\partial_xf\left (x\right )+\partial_xf\left (x-\eta\right )\Xi_2(x,\eta,f)\bigg{]}d\eta=\frac{\rho^2-\rho^1}{4l}A[f](x),
\label{eq0.1}
\end{multline}
where the singular kernels $\Xi_1$ and $\Xi_2$ are defined as 
\begin{equation}
\Xi_1(x,\eta)=\frac{\sinh\left(\frac{\pi}{2l}\eta\right)}{\cosh \left(\frac{\pi}{2l}\eta\right)-\cos(\frac{\pi}{2l}(f(x)-f(x-\eta)))},
\label{Xi1}
\end{equation}
corresponding to the singular character of the problem, and 
\begin{equation}
\Xi_2(x,\eta)=\frac{\sinh\left(\frac{\pi}{2l}\eta\right)}{\cosh \left(\frac{\pi}{2l}\eta\right)+\cos(\frac{\pi}{2l}(f(x)+f(x-\eta)))},
\label{Xi2}
\end{equation}
which becomes singular when $f$ reaches the boundaries. The text $\text{P.V.}$ denotes principal value. As for the whole plane case (see \cite{ccgs-10, c-g-o08}) the spatial operator $A[f](x)$ can be written as an $x$-derivative. Indeed, 
\begin{multline}\label{eqderiv}
A[f](x)=\frac{2l}{\pi}\text{P.V.}\int_\RR\partial_x\left(\arctan\left(\frac{\tan\left(\frac{\pi}{2l}\frac{f(x)-f(x-\eta)}{2}\right)}{\tanh\left(\frac{\pi}{2l}\frac{\eta}{2}\right)}\right)\right)d\eta\\
+\frac{2l}{\pi}\text{P.V.}\int_\RR\partial_x\left(\arctan\left(\tan\left(\frac{\pi}{2l}\frac{f(x)+f(x-\eta)}{2}\right)\tanh\left(\frac{\pi}{2l}\frac{\eta}{2}\right)\right)\right)d\eta
\end{multline}
and we conclude the mean conservation
\begin{equation}\label{mean}
\int_\RR f(x,t)dx=\int_{\RR}f_0(x)dx.
\end{equation}

When we do not parametrize the curve as a graph, \emph{i.e.}, we consider $z(\alpha)=(z_1(\alpha),z_2(\alpha))$, we obtain the equation
\begin{multline}\label{eqcurva}
\pat z= \frac{\rho^2-\rho^1}{4\pi}\text{P.V.}\int_\RR\left[\frac{(\partial_\alpha z(\alpha)-\partial_\alpha z(\eta))\sinh(z_1(\alpha)-z_1(\eta))}{\cosh(z_1(\alpha)-z_1(\eta))-\cos(z_2(\alpha)-z_2(\eta))}\right.\\
+\left.\frac{(\partial_\alpha z_1(\alpha)-\partial_\alpha z_1(\eta),\partial_\alpha z_2(\alpha)+\partial_\alpha z_2(\eta))\sinh(z_1(\alpha)-z_1(\eta))}{\cosh(z_1(\alpha)-z_1(\eta))+\cos(z_2(\alpha)+z_2(\eta))}\right]d\eta.
\end{multline}

As our interface moves in a bounded medium the correct space to consider is
$$
H^s_l=H^s(\RR)\cap \{f:\|f\|_{L^\infty}<l\}.
$$

The density $\rho$ defined as in \eqref{eq3a} is a weak solution of the conservation of mass equation present in \eqref{eq4} if and only if the interface verifies the equation \eqref{eq5} (see Proposition \ref{consermass} in Section \ref{sec1} below). It also follows (see \ref{limiota}) that if we take the limit $\mathcal{A}\rightarrow0$ we recover the equation \eqref{full} (see \cite{c-g07}).

In a recent work \cite{e-m10}, J. Escher and B-V.Matioc studied the problem \eqref{eq6}, \eqref{eq7} in the case with different viscosities and surface tension in a periodic (in $x$) domain when $0<\mathcal{A}<1$. They obtained an abstract evolution equation for the interface and showed well-posedness in the classical sense when the Rayleigh-Taylor condition is satisfied and the interface is in a neighbourhood of the zero function in certain H\"{o}lder spaces. They consider the problem as a problem in two coupled domains. The domains are coupled by the interface and by the Laplace-Young condition
$$
p^2(x,f(x,t),t)-p^1(x,f(x,t),t)=\gamma \kappa[f],
$$
where $\kappa[f]$ denotes the curvature of the interface $f(x,t)$ and $\gamma$ denotes the surface tension coefficient.

In Section \ref{sec2} we study the similarities between the case $\mathcal{A}=0$ and $0<\mathcal{A}<1$. First, we prove local well-posedness in Sobolev spaces (see Section \ref{sec2.1}) and instant analyticity in a growing complex strip when the Rayleigh-Taylor condition is achieved (see Section \ref{sec2.2}). The last similarity studied in Section \ref{sec2.3} is that for arbitrary initial curves which are analytic there is an unique local solution, which is analytic, both forward and backward in time. We remark that for this result the Rayleigh-Taylor condition is not needed. The proofs follows the steps of the paper \cite{ccfgl}. Here we show that the contribution from the boundary does not affect the \emph{a priori} estimates from \cite{ccfgl}. In Section \ref{sec2.4} we show an ill-posedness result in Sobolev spaces. The key point of this result is that we do not need global existence for some class of solutions to prove the result (compare with \cite{c-g07} and \cite{SCH}).

The main purpose of this work is to study the differences between the case with infinite depth and the case with bounded medium. This is done in Section \ref{sec3}, where we study some qualitative properties of the solution. We prove the maximum principle for $\|f(t)\|_{L^\infty}$ and also for $\|\pax f\|_{L^\infty}$ by studying the evolution of the maximum or the minimum values. The ODEs coming from these analysis have local and non-local terms and the main dificulty is to compare these two different kind of terms in order to ensure the decay. In particular, we prove the following decay estimate for $\|f(t)\|_{L^\infty}$ in a confined medium:
\begin{equation}\label{decaydesi}
\frac{d}{dt}\|f(t)\|_{L^\infty}\leq-c(\|f_0\|_{L^1},\|f_0\|_{L^\infty},\rho^2,\rho^1,l)e^{-\frac{\pi\|f_0\|_{L^1}}{l\|f(t)\|_{L^\infty}}}.
\end{equation}
As a corollary we conclude that the unique one-signed, integrable, stationary solution is the rest state. Let us observe that the natural boundary condition for the velocity, $v\cdot n=0$, imposes that if our initial interface is close enough to the boundary the evolution of the maximum is very slow. Due to this fact, we obtain the slow decay inequality \eqref{decaydesi}. 

We show that if the initial data is in a region depending on $\|f_0\|_{L^\infty}$ and $\|\pax f_0\|_{L^\infty}$ then we have a maximum principle for $\|\pax f\|_{L^\infty}$ in a confined medium.
Indeed, we consider a smooth initial data $f_0$ in the Rayleigh-Taylor stable regime such that the following conditions holds:
\begin{equation}\label{H3}
\|\pax f_0\|_{L^\infty}\leq1,
\end{equation}
\begin{equation}\label{H4}
\tan\left(\frac{\pi\|f_0\|_{L^\infty}}{2l}\right)< \|\pax f_0\|_{L^\infty}\tanh\left(\frac{\pi}{4l}\right),
\end{equation}
and
\begin{multline}\label{H5}
\left(\|\pax f_0\|_{L^\infty}+|2(\cos\left(\frac{\pi}{2l}\right)-2)\sec^4\left(\frac{\pi}{4l}\right)|\|\pax f_0\|_{L^\infty}^3\right)\frac{\pi^3}{8l^3}\\
\times\frac{\left(1+\|\pax f_0\|_{L^\infty}\left(\|\pax f_0\|_{L^\infty}+\frac{\tan\left(\frac{\pi }{2l}\frac{\|\pax f_0\|_{L^\infty}}{2}\right)}{\tanh\left(\frac{\pi }{4l}\right)}\right)\right)}{6\tanh\left(\frac{\pi }{4l}\right)}\frac{\pi^2}{4l^2}\\
+4\tan\left(\frac{\pi }{2l}\|f_0\|_{L^\infty}\right)-4\|\pax f_0\|_{L^\infty}\cos\left(\frac{\pi }{l}\|f_0\|_{L^\infty}\right)\leq0.
\end{multline}

Moreover, if $(x(l),y(l))$ is the solution of the system
\begin{equation}\label{sisder}
\left\{ \begin{array}{ll}
         \tan\left(\frac{\pi x}{2l}\right)- y\tanh\left(\frac{\pi}{4l}\right)=0\\
        \left(y+|2(\cos\left(\frac{\pi}{2l}\right)-2)\sec^4\left(\frac{\pi}{4l}\right)|y^3\right)\frac{\left(1+y\left(y+\frac{\tan\left(\frac{\pi }{2l}\frac{y}{2}\right)}{\tanh\left(\frac{\pi }{4l}\right)}\right)\right)}{6\tanh\left(\frac{\pi }{4l}\right)}\left(\frac{\pi}{2l}\right)^5\\
        \qquad\qquad\qquad\qquad\qquad\qquad\qquad+4\tan\left(\frac{\pi }{2l}x\right)-4y\cos\left(\frac{\pi }{l}x\right)=0,
        \end{array}\right.
\end{equation}
and we have that $\|\pax f_0\|_{L^\infty}<y(l)$ and $\|f_0\|_{L^\infty}<x(l)$ we obtain 
$$
\|\pax f\|_{L^\infty}\leq 1.
$$

\setcounter{figure}{1}
\begin{figure}[t]
		\begin{center}
		\includegraphics[scale=0.3]{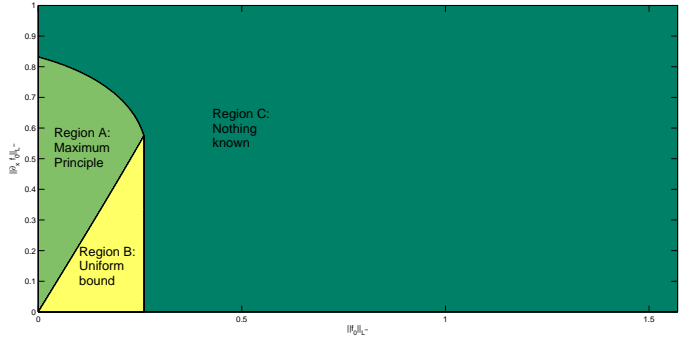} 
		\end{center}
\caption[Region in $(\|f_0\|_{L^\infty},\|\pax f_0\|_{L^\infty})$]{Different regions in $(\|f_0\|_{L^\infty},\|\pax f_0\|_{L^\infty})$ for the behaviour of $\|\pax f\|_{L^\infty}$ when $\pi=2l$.}		
		\label{2d_0}
\end{figure}

The effect of the boundaries is very important at this level, and we obtain a region (Region B in Figure~\ref{2d_0}) where we do not have maximum principle for $\|\pax f_0\|$ but we have an uniform bound $\|\pax f (t)\|_{L^\infty}\leq 1\;\forall t\geq0$. The region A is the region where we have maximum principle for the derivative. Due to the term coming from the effect of the boundaries, $\Xi_2$, the conditions that we obtain are much more restrictives than in the deep water regime (the case with infinite depth). The previous result gives us conditions on the smallness of $\mathcal{A}$ and $\|\pax f_0\|_{L^\infty}$ (which, for fixed amplitude, can be understood as the \emph{'wavelength'} of the wave) so, roughly speaking the Theorem says that if we are in the long wave regime (small amplitude and large wavelenght) then there is no \emph{turning effect}, \emph{i.e.} there are no shocks. We remark that if we take the limit $\mathcal{A}\rightarrow0$ we recover the result for the deep water regime 
contained in \cite{c-g09}

We study the formation of singularities in Section \ref{sec3.2}. The singularity is a blow up of $\|\pax f (t)\|_{L^\infty}$. Physically this result means that there are waves such that they \emph{'turn over'}. 
Moreover, we compare this result with the result for the deep water regime (see \cite{ccfgl}). In particular, we obtain firm numerical evidence of the following turning effect in a confined medium: There exist initial data $z_0(\alpha,0)=(z_1(\alpha,0),z_2(\alpha,0))$ such that in finite time the solution of \eqref{eqcurva} achieves the unstable regime only when the depth is finite. If the depth is infinite the same curves become graphs (see Section \ref{sec3.3}).

\begin{coment}
In order to simplify the notation we take $\mu/\kappa=g=1$ and we sometimes suppress the dependence on $t$. We denote $v_i$ the component $i-$th of the vector $v$. We remark that $v^i$ is the velocity field in $S^i(t)$. We write $n$ for the unitary normal to the curve $\Gamma$ vector and $\bar{n}$ for the non-unitary normal vector. We denote $\bar{\rho}=\frac{\rho^2-\rho^1}{4l}$. In the rest of the paper we take, without loss of generality, $2l=\pi$ and $\bar{\rho}=2$ if there is no other explicit statement.
\end{coment}


\section{The equation for the internal wave}\label{sec1}
In this section we obtain the equation for the interface $z(\alpha,t)$ in an explicit formula. First we have to add impermeable boundary conditions for $v$, $\emph{i.e.}$ $v(x,\pm l,t)\cdot n=0$. 

Using the incompressibility condition we have that there exists a scalar function $\Psi$ such that $v=\nabla^\perp \Psi$. The function $\Psi$ is the stream function. Then 
\begin{equation*}\label{pois}
\Delta \Psi=-\text{curl} (0,\rho)=\omega
\end{equation*}
where the vorticity is supported on the curve
$$
\omega(x,y)=\varpi(\alpha)\delta((x,y)-z(\alpha,t)),
$$
with amplitude
$$
\varpi(\alpha)=-(\rho^2-\rho^1)\paa z_2(\alpha).
$$
In this domain we need to obtain the Biot-Savart law. The Green function for the equation $\Delta u=f$ in the strip $\RR\times(0,2l)$ (with homogeneous Dirichlet conditions) is given by the convolution with the kernel
\begin{multline*}
G(x,y,\mu,\nu)=\frac{1}{2\pi}\sum_{n=-\infty}^\infty\bigg{[}\log\left(\sqrt{(x-\mu)^2+(y-(4nl+\nu))^2}\right)\\
-\log\left(\sqrt{(x-\mu)^2+(y-(4nl-\nu))^2}\right)\bigg{]}.
\end{multline*}
The Biot-Savart law in this strip is given by the kernel
$$
BS(x,y,\mu,\nu)=\nabla_{x,y}^\perp G(x,y,\mu,\nu)=\frac{1}{2\pi}\sum_{n=-\infty}^\infty\bigg{[}\frac{(\Gamma_n^+)^\perp}{|\Gamma_n^+|^2}-\frac{(\Gamma_n^-)^\perp}{|\Gamma_n^-|^2}\bigg{]},
$$
where 
$$
\Gamma_n^+=(x-\mu,y-(4nl+\nu)),\;\;\Gamma_n^-=(x-\mu,y-(4nl-\nu)).
$$
It is useful to consider complex variables notation. Then
$$
\overline{BS}(x,y,\mu,\nu)=\frac{1}{2\pi i}\sum_{n=-\infty}^\infty\bigg{[}\frac{1}{\Gamma_n^+}-\frac{1}{\Gamma_n^-}\bigg{]}.
$$
Fixed $n$, we compute the following
$$
\frac{1}{\Gamma_n^+}+\frac{1}{\Gamma_{-n}^+}=\frac{2(x-\mu+i(y-\nu))}{(x-\mu+i(y-\nu))^2+(4nl)^2},
$$
$$
\frac{1}{\Gamma_n^-}+\frac{1}{\Gamma_{-n}^-}=\frac{2(x-\mu+i(y+\nu))}{(x-\mu+i(y+\nu))^2+(4nl)^2}.
$$
We change variables ($y-l=y, \nu-l=\nu$) to recover the initial strip $S=\RR\times(-l,l)$, moreover, without lossing generality we take $l=\pi/2$. Due to the formula
\begin{equation*}
\frac{1}{z}+\sum_{n= 1}^\infty \frac{2z}{z^2+(2n\pi)^2}=\frac{1}{2}\coth\left(\frac{z}{2}\right),
\end{equation*}
we obtain that the Biot-Savart law in cartesian coordinates is given by
\begin{multline*}
BS(x,y,\mu,\nu) =\frac{1}{4\pi}\left(-\frac{\sin(y-\nu)}{\cosh(x-\mu)-\cos(y-\nu)}-\frac{\sin(y+\nu)}{\cosh(x-\mu)+\cos(y+\nu)},\right.\\
\left.\frac{\sinh(x-\mu)}{\cosh(x-\mu)-\cos(y-\nu)}-\frac{\sinh(x-\mu)}{\cosh(x-\mu)+\cos(y+\nu)}\right).
\end{multline*}
Using the formula for the vorticity we have that the velocity is
\begin{multline*}
v(x,y)=\int_\RR\left(\frac{-(4\pi)^{-1}\varpi(\beta)\sin(y-z_2(\beta))}{\cosh(x-z_1(\beta))-\cos(y-z_2(\beta))}+\frac{-(4\pi)^{-1}\varpi(\beta)\sin(y+z_2(\beta))}{\cosh(x-z_1(\beta))+\cos(y+z_2(\beta))},\right.\\
\left.\frac{(4\pi)^{-1}\varpi(\beta)\sinh(x-z_1(\beta))}{\cosh(x-z_1(\beta))-\cos(y-z_2(\beta))}-\frac{(4\pi)^{-1}\varpi(\beta)\sinh(x-z_1(\beta))}{\cosh(x-z_1(\beta))+\cos(y+z_2(\beta))}\right)d \beta.
\end{multline*}
We use the identity
$$
\int_\RR \partial_\eta \log(\cosh(z_1(\alpha)-z_1(\eta))\pm\cos(z_2(\alpha)\pm z_2(\eta)))=0
$$
to obtain that the average velocity in the curve is
\begin{multline*}
v(z(\alpha))=\left(-\frac{\bar{\rho}}{2}\text{P.V.}\int_\RR\partial_\alpha z_1(\eta)\bigg{[}\frac{\sinh(z_1(\alpha)-z_1(\eta))}{\cosh(z_1(\alpha)-z_1(\eta))-\cos(z_2(\alpha)-z_2(\eta))}\right.\\
\qquad+\frac{\sinh(z_1(\alpha)-z_1(\eta))}{\cosh(z_1(\alpha)-z_1(\eta))+\cos(z_2(\alpha)+z_2(\eta))}\bigg{]}d\eta,\\
-\frac{\bar{\rho}}{2}\text{P.V.}\int_\RR\partial_\alpha z_2(\eta)\bigg{[}\frac{\sinh(z_1(\alpha)-z_1(\eta))}{\cosh(z_1(\alpha)-z_1(\eta))-\cos(z_2(\alpha)-z_2(\eta))}\\
\left.\qquad-\frac{\sinh(z_1(\alpha)-z_1(\eta))}{\cosh(z_1(\alpha)-z_1(\eta))+\cos(z_2(\alpha)+z_2(\eta))}\bigg{]}d\eta\right).
\end{multline*}

The interface is convected by this velocity but we can add any velocity in the tangential direction without altering the shape of the curve. The tangential velocity in a curve only changes the parametrization. We consider then the following equation with the redefined velocity
$$
\pat z(\alpha)=v(z(\alpha))+c(\alpha)\partial_\alpha z(\alpha),
$$
where
\begin{multline*}
c(\alpha)=\frac{\bar{\rho}}{2}\text{P.V.}\int_\RR\frac{\sinh(z_1(\alpha)-z_1(\eta))}{\cosh(z_1(\alpha)-z_1(\eta))-\cos(z_2(\alpha)-z_2(\eta))}\\
+\frac{\sinh(z_1(\alpha)-z_1(\eta))}{\cosh(z_1(\alpha)-z_1(\eta))+\cos(z_2(\alpha)+z_2(\eta))}d\eta.
\end{multline*}
Following this approach we obtain \eqref{eqcurva}. Because of that choice of $c(\alpha)$ we obtain that, if initially the curve can be parametrized as a graph, \emph{i.e.}, $z(x,0)=(x,f_0(x))$, we have that the velocity $v_1$ on the curve is zero, thus our curve is parametrized as a graph for $t>0$ and we recover the contour equation \eqref{eq0.1}.

Note that when $l\rightarrow\infty$ in the equation \eqref{eq0.1} we recover the equation for the whole plane \eqref{full}: 
\begin{equation}\label{limiota}
\lim\limits_{l\to\infty}\frac{1}{l} A[f](x)=\frac{2}{\pi}\text{P.V.}\int_\RR \frac{\eta(\pax f(x)-\pax f(x-\eta))}{\eta^2+(f(x)-f(x-\eta))^2}d\eta,
\end{equation}
where $A[f]$ is the operator defined in \eqref{eq0.1}.

Furthermore, we obtain the pressure $p$ (up to a constant) solving the equation
$$
-\Delta p=g\partial_y\rho,
$$
with Neumann boundary conditions $$\partial_n p|_{y=l}=-g\rho^1,\qquad\partial_n p|_{y=-l}=g\rho^2.$$

In this way we obtain $v,p$ satisfying Darcy's Law and the incompressibility condition. It is easy to check that $\rho(x,y,t)$ is a weak solution of the conservation of mass equation.

\begin{defi}
Let $v$ be an incompressible field of velocities following Darcy's Law. We define the weak solution of the conservation of mass equation present in \eqref{eq4} as a function satisfying
\begin{equation*}
\int_0^T\int_{\RR}\int_{-l}^{l}\rho(x,y,t)\partial_t \phi(x,y,t)+v(x,y,t)\rho(x,y,t)\nabla_{x,y}\phi (x,y,t)dydxdt=0  
\end{equation*}
for all $\phi\in C^{\infty}_c(\RR\times(-l,l)\times(0,T))$.
\label{def1}
\end{defi}

We conclude this section with the following result.

\begin{prop}\label{consermass}
Let $\rho$ be the function defined in \eqref{eq3a}. Then $\rho$ is a weak solution of the conservation of mass equation (see Definition \ref{def1}) if and only if $f$ is a solution of \eqref{eq0.1}. 
\end{prop}

The proofs of these two results are straightforward and, for the sake of brevity, we left them for the interested reader.


\section{Similar results between the two regimes}\label{sec2}
In this section we show a group of results for $0<\mathcal{A}<1$ similar to those in the regime $\mathcal{A}=0$. Some proofs follow the same ideas but, due to the structure in our equation \eqref{eq0.1}, with a second term coming from the boundaries present in our model, there are some difficulties. We show the well-posedness in Sobolev spaces when the Rayleigh-Taylor condition is satified, \emph{i.e.} the denser fluid is below the lighter one. In the case where the Rayleigh-Taylor condition is not satisfied but our initial data is analytic we also have a well-posedness result by means of a Cauchy-Kovalevski Theorem. We prove the smoothing effect of the spatial operator in \eqref{eq0.1}, \emph{i.e.} the solution becomes instantly analytic. We also apply this smoothing effect to prove an ill-posedness result when the system is in the unstable regime.

\subsection{Well-posedness in Sobolev spaces}\label{sec2.1}
In this section we sketch the proof of local well-posedness in Sobolev spaces in the Rayleigh-Taylor stable case:

\begin{thm}
If the Rayleigh-Taylor condition is satisfied, \emph{i.e.} $\rho^2-\rho^1>0$, and the initial data $f_0(x)=f(x,0)\in H_l^k(\RR)$, $k\geq3$, then there exists an unique classical solution of \eqref{eq0.1} $f\in C([0,T],H_l^k(\RR))$  where $T=T(\|f_0\|_{H^k}, \|f_0\|_{L^\infty})$. Moreover, we have $f\in C^1([0,T],C(\RR))\cap C([0,T],C^2(\RR)).$
\label{teo1}
\end{thm}
\begin{proof}
We indicate the constants with a dependency on $l$ as $c(l)$. The proof follows the same lines as in \cite{c-g07}, \emph{i.e.} we obtain \emph{a priori} bounds for the appropriate energy which allow us to regularize the system and to take the limit of the regularized solutions. In order to deal with the kernel $\Xi_2$, the kernel corresponding to the effect of the boundaries, we define the following energy:
\begin{equation}
E[f](t)=\|f\|^2_{H^3}(t)+\|d[f]\|_{L^\infty}(t),
\label{energy} 
\end{equation}
where $d[f]:\RR^2\times \RR^+\mapsto \RR^+$ is defined as
\begin{equation}
d[f](x,\eta,t)=\frac{1}{\cosh(\eta)+\cos(f(x)+f(x-\eta))}. 
\label{distancia}
\end{equation}
The function \eqref{distancia} measures the distance between $f$ and the top and floor $\pm l$. In other words, $\|d[f]\|_{L^\infty}<\infty$ implies that $\|f\|_{L^\infty}<\frac{\pi}{2}$. So this is the natural \emph{'energy'} associated to the space $H^3_{l}(\RR)$. We obtain \emph{'a priori'} energy estimates as in \cite{c-g07}. The integrals corresponding to the kernel $\Xi_1$ are the more singular terms and can be bounded as in \cite{c-g07} because has a singularity with the same order. Indeed, we compute
\begin{multline*}
\Xi_1=\frac{\sinh(\eta)-\eta}{\cosh(\eta)-\cos(f(x)-f(x-\eta))}\\
+\frac{\eta}{\cosh(\eta)-\cos(f(x)-f(x-\eta))}-\frac{2\eta}{\eta^2\left(1+\left(\frac{f(x)-f(x-\eta)}{\eta}\right)^2\right)}\\+\frac{2\eta}{\eta^2+\left(f(x)-f(x-\eta)\right)^2}.
\end{multline*}
The last term in this expression is (up to a constant) the kernel obtained when the fluids fill the whole plane and the other terms are not singular.

The integrals corresponding to the kernel $\Xi_2$ are harmless and can be bounded using the definition of $d[f]$. For instance, one of the integrals arising in the study of the third derivative, after an integration by parts, is
$$
I=\int_{\RR}|\pax^3f(x)|^2\left(\int_{B(0,1)}+\int_{B^c(0,1)}\right)\pax\Xi_2(x,\eta)d\eta dx=I_{in}+I_{out},
$$
and we obtain
\begin{eqnarray*}\nonumber
I_{in}&\leq&\int_{\RR}|\pax^3f(x)|^2\text{P.V.}\int_{B(0,1)}\sinh\left( |\eta|\right)\bigg{|}\frac{\sin\left( f(x)+f(x-\eta)\right) (\pax f(x)+\pax f(x-\eta))}{\left(\cosh \left( \eta\right)+\cos( (f(x)+f(x-\eta)))\right)^2}\bigg{|}d\eta dx\\ 
&\leq&c(l)\|f\|_{C^1}\|\pax^3f\|_{L^2}^2\|d[f]\|^2_{L^\infty}
\leq c(l)\|f\|_{H^3}^3\|d[f]\|^2_{L^\infty}, 
\end{eqnarray*}
\begin{equation*}
I_{out}\leq c(l)\|f\|_{C^1}\|f\|_{H^3}^2\int_{B^c(0,1)}\frac{\sinh\left( |\eta|\right)}{(\cosh\left( \eta\right)-1)^2}d\eta\leq c(l)\|f\|_{H^3}^3.
\end{equation*}
With these techniques we obtain
$$
\frac{d}{dt}\|f\|_{H^3}\leq c(l)(E[f]+1)^5.
$$

In order to use classical energy methods we have to bound the evolution of $\|d[f]\|_{L^\infty}$ in terms of the energy $E[f]$. With this method we need a bound on $\|\pat f\|_{L^\infty}$. In order to do this we split $\pat f$ in two terms, one for each kernel
$$
\pat f=A_1+A_2.
$$
We give the proof for the 
$$
A_1=\text{P.V.}\int_\RR(\pax f(x)-\pax f(x-\eta))\Xi_1(x,\eta)d\eta.
$$ 
For the term corresponding to the second kernel, $A_2$, the procedure is analogous.

We split $A_1$ in its \emph{'in'} and \emph{'out'} parts, $A_1=A^{in}_1+A^{out}_1$, with
$$
A^{in}_1\leq c(l)\|\pax^2f\|_{L^\infty},
$$ 
and
$$
A^{out}_1\leq\left\|\text{P.V.}\int_{B^c(0,1)}\pax f(x)\Xi_1(x,\eta)d\eta\right\|_{L^\infty}+\left\|\text{P.V.}\int_{B^c(0,1)}-\pax f(x-\eta)\Xi_1(x,\eta)d\eta\right\|_{L^\infty}.
$$
We have that the integral
$$
\text{P.V.}\int_{B^c(0,1)}\frac{\sinh(\eta)}{\sinh^2\left(\frac{\eta}{2}\right)}d\eta=0,
$$ 
using this fact and the classical and hyperbolic trigonometric identities we can write
$$
\text{P.V.}\int_{B^c(0,1)}\Xi_1(x,\eta)d\eta=\text{P.V.}\int_{B^c(0,1)}\frac{\sinh(\eta)}{2\sinh^2\left(\frac{\eta}{2}\right)}\cdot\left(\frac{1}{1+\frac{\sin^2((f(x)-f(x-\eta))/2)}{\sinh^2(\eta/2)}}-1\right)d\eta.
$$
We compute
$$
\text{P.V.}\int_{B^c(0,1)}\Xi_1(x,\eta)d\eta=\text{P.V.}\int_{B^c(0,1)}\frac{\sinh(\eta)}{2\sinh^2\left(\frac{\eta}{2}\right)}\cdot\left(\frac{\frac{-\sin^2((f(x)-f(x-\eta))/2)}{\sinh^2(\eta/2)}}{1+\frac{\sin^2((f(x)-f(x-\eta))/2)}{\sinh^2(\eta/2)}}\right)d\eta,
$$
and we obtain
$$
\left\|\text{P.V.}\int_{B^c(0,1)}\pax f(x)\Xi_1(x,\eta)d\eta\right\|_{L^\infty}\leq c(l)\|\pax f\|_{L^\infty}.
$$
We note that
$$
\text{P.V.}\int_{B^c(0,1)}-\pax f(x-\eta)\Xi_1(x,\eta)d\eta=\text{P.V.}\int_{B^c(0,1)}\partial_\eta f(x-\eta)\Xi_1(x,\eta)d\eta.
$$
In order to bound this integral we integrate by parts. We conclude
$$
\left\|\text{P.V.}\int_{B^c(0,1)}-\pax f(x-\eta)\Xi_1(x,\eta)d\eta\right\|_{L^\infty}\leq c(l)\|f\|_{L^\infty}(1+\|\pax f\|_{L^\infty}).
$$
Thus we get
$$
\|\pat f\|_{L^\infty}\leq c(l)\|f\|_{L^\infty}(1+\|\pax f\|_{L^\infty}).
$$
Now, we can prove the last estimate. We have that 
$$
\frac{d}{dt}d[f]=d[f]^2\sin(f(x)+f(x-\eta))(\pat f(x)+\pat f(x-\eta))\leq c(l)d[f]^2\|\pat f\|_{L^\infty}.
$$ 
Due to the previous bound for $\|\pat f\|_{L^\infty}$, we obtain
$$
\frac{d}{dt}d[f]\leq c(l)d[f] \|d[f]\|_{L^\infty}(E[f]+1)^2.
$$ 
Integrating in time, we get
$$
d[f](t+h)\leq d[f](t)e^{\int_t^{t+h}c(l)\|d[f]\|_{L^\infty}(E[f](s)+1)^2ds}.
$$
Finally we have 
$$
\frac{d}{dt}\|d[f]\|_{L^\infty}=\lim_{h\rightarrow0}\frac{\|d[f]\|_{L^\infty}(t+h)-\|d[f]\|_{L^\infty}(t)}{h}\leq c(l)(E[f]+1)^4.
$$
Putting all together we obtain the following bound
\begin{equation*}
\frac{d}{dt}E[f]\leq c(l)(E[f]+1)^{5},
\end{equation*}
therefore
$$
E[f](t)\leq \frac{E[f_0]}{\sqrt[4]{-4E[f_0]^4c(l)t+1}}.
$$
Now we regularize equation \eqref{eq0.1} in the classical way using mollifiers (see \cite{bertozzi-Majda}) and these regularized equations have an unique classical solution. The estimates for the regularized equations mimic the previous ones above, thus with these \emph{'a priori'} bounds we can obtain the local existence by taking the limit solutions of the regularized equations. The proof of uniqueness of classical solutions follows the same ideas.
\end{proof}

\subsection{Smoothing effect}\label{sec2.2}
In this section we sketch the proof of the instant analyticity for the classical solution (which exists due to the result in the previous Section).

\begin{thm}\label{IA}
Let $f_0\in H^3_l(\RR)$ be the initial data and the Rayleigh-Taylor condition is satisfied then the unique classical solution $f(x,t)$ to equation \eqref{eq0.1} continues analitically into the strip $\BB=\{x+i\xi,|\xi|<kt, \forall \,0<t\leq T(f_0)\}$ with $k=k(f_0).$
\end{thm}

\begin{proof}
The proof, as the one in \cite{ccfgl}, relies on some \emph{a priori} estimates for the complex extension of the function $f$ on the boundary of the strip 
$$
\BB=\{x+i\xi,|\xi|<kt, \forall \,\,0<t\leq T(f_0)\}
$$ 
for certain $k$, a constant that will be fixed later. Once the evolution of the appropriate energy is bounded, we construct regularized equations with analytical solutions and such that the same estimates hold. Therefore we can pass to the limit in the regularized solutions to obtain a solution of the original problem. We denote
$$
\|f\|^2_{L^2(\BB)}=\int_\RR |f(x+ikt)|^2dx+\int_\RR |f(x-ikt)|^2dx,\;\;\|f\|^2_{H^3(\BB)}=\|f\|^2_{L^2(\BB)}+\|\pax^3 f\|^2_{L^2(\BB)},
$$
\begin{equation*}
d^+[f](x+i\xi,\eta)=\frac{\cosh^2(\eta/3)}{\cosh(\eta)+\cos(f(x+i\xi)+f(x+i\xi-\eta))},
\end{equation*}
\begin{equation*}
d^-[f](x+i\xi,\eta)=\frac{\sinh^2(\eta/3)}{\cosh(\eta)-\cos(f(x+i\xi)-f(x+i\xi-\eta))},
\end{equation*}
and
\begin{equation}\label{m}
m(t)=\min_\gamma\re\frac{1}{1+(\pax f(\gamma))^2}
\end{equation}
We remark that $d^+[f]$, as $d[f]$ defined in \eqref{distancia}, controls the distance to the boundaries. $d^-[f]$ plays the role of the arc-chord condition (see \cite{ccfgl}) and ensures that the singularity in the first kernel has order two. In order to get energy estimates working with complex functions, we need to study when the kernels $\Xi_1$ and $\Xi_2$ are singular. If $\Xi_1$ is singular then the following equality holds
$$
\cosh(\eta)-\cos(\re [f(x+ikt)- f(x+ikt-\eta)])\cosh(\im [f(x+ikt)- f(x+ikt-\eta)])=0.
$$ 
Assume now that $|\eta|>R\geq2l$, where $R>>2l$ is a fixed constant, then
\begin{multline*}
\frac{1}{2}\cosh(\eta)+\frac{1}{2}\cosh(R)-\cosh(2\| f\|_{L^\infty(\BB)})\\
\leq \frac{1}{2}\cosh(\eta)+\frac{1}{2}\cosh(R)-\cosh(2\|\im f\|_{L^\infty(\BB)})\\
\leq\cosh(\eta)-\cos(\re [f(x+ikt)-f(x+ikt-\eta)])\cosh(\im [f(x+ikt)- f(x+ikt-\eta)]).
\end{multline*}
Then, taking $R>>2l$ such that
$$
\frac{1}{2}\cosh(R)-\cosh(2\|f\|_{L^\infty(\BB)})\geq0,
$$
we ensure that the kernel $\Xi_1$ is not singular in this region. A similar analysis can be done for $\Xi_2$ to obtain the same condition. 

We denote
$$
D[f](\gamma)=\frac{1}{\cosh(R)-2\cosh(2|f(\gamma)|)}.
$$
This term is a technical resource to obtain enough decay at infinity. We consider Hardy-Sobolev spaces (see \cite{bakan2007hardy} and references therein) on $\BB$ so we want to obtain \emph{'a priori'} bounds on the following energy
\begin{equation}\label{enIA}
E_\BB[f]=\|f\|^2_{H^3(\BB)}+\|d^+[f]\|_{L^\infty(\BB))}+\|d^-[f]\|_{L^\infty(\BB)}+\|D[f]\|_{L^\infty(\BB)}, 
\end{equation}
where
$$
\|F(x+i\xi,\eta)\|_{L^\infty(\BB)}=\sup_{x+i\xi\in \BB,\\ \eta\in \RR} |F(x+i\xi,\eta)|.
$$
The evolution for the complex extension of $f$ is
\begin{multline}
\pat f(x\pm ikt) =\text{P.V.}\int_\RR\bigg{[}\frac{\left (\partial_xf\left (x\pm ikt\right )-\partial_xf\left (x\pm ikt-\eta\right )\right)\sinh\left(\eta\right)}{\cosh \left(\eta\right)-\cos((f(x\pm ikt)-f(x\pm ikt-\eta)))}\\
+ \frac{(\partial_xf\left (x\pm ikt\right )+\partial_xf\left (x\pm ikt-\eta\right )\sinh\left(\eta\right)}{\cosh \left(\eta\right)+\cos((f(x\pm ikt)+f(x\pm ikt-\eta)))}\bigg{]}d\eta.
\label{eq0.1IA}
\end{multline}
Recall that $R>>2l$ is a fixed constant. Then, as in \cite{c-g07, ccfgl}, we obtain
\begin{multline}\label{tojunto}
\frac{d}{dt}E_{\BB}[f]\leq \exp(c(l)(E_\BB[f]+1))\\
+\|\Lambda^{1/2}\pax^3 f\|^2_{L^2(\BB)}\left(K\left\|\im\left(\text{P.V.}\int_\RR \Xi_1+\Xi_2\right)\right\|_{L^\infty(\BB)}+2k-2\pi m(t)\right),
\end{multline}
where $K=K(l)$ is an universal constant and $k$ is the width of the strip. At time $t=0$, we have that if $k$ is small enough ($k$ only depends on the initial data) 
\begin{multline}\label{kcond}
4k+K\left\|\im\left(\text{P.V.}\int_\RR \Xi_1+\Xi_2\right)\bigg{|}_{t=0}\right\|_{L^\infty(\BB)}-\frac{2\pi}{1+\|\pax f_0\|^2_{L^\infty(\RR)}}\\
=4k-\frac{2\pi}{1+\|\pax f_0\|^2_{L^\infty(\RR)}}<0.
\end{multline}
Then we need to show that this quantity remains negative (at least) for a short time. In order to do this we define the following new energy:
\begin{equation}\label{energy2IA}
\mathfrak{E}_{\BB}[f]=E_\BB[f] +\frac{1}{2\pi m(t)-K\left\|\im\left(\text{P.V.}\int_\RR \Xi_1+\Xi_2\right)\right\|_{L^\infty(\BB)}-4k}.
\end{equation}
If $\mathfrak{E}_\BB[f]<\infty$ then $\frac{d}{dt}E_\BB[f]\leq \exp(c(l)(E_\BB[f]+1))$ and we have the correct \emph{'a priori'} estimates.

We need to bound $m'(t)$ and $\frac{d}{dt}\left\|\im\left(\text{P.V.}\int_\RR \Xi_1+\Xi_2d\eta\right)\right\|_{L^\infty(\BB)}$. Now, if we have a classical solution with $E_\BB[f]<\infty$, the Sobolev embedding gives us that 
$$
\re\frac{1}{1+(\pax f(\gamma))^2}\in C^1([0,T]\times\BB)
$$ 
Thus we can apply Rademacher Theorem to $m(t)$ in \eqref{m} and we get
$$
m'\leq \exp(c(l)(\mathfrak{E}_\BB[f]+1))
$$ 
Again, applying Rademacher Theorem, we have
$$
\frac{d}{dt}\left\|\im\left(\text{P.V.}\int_\RR \Xi_1+\Xi_2d\eta\right)\right\|_{L^\infty(\BB)}\leq \exp(c(l)(\mathfrak{E}_\BB[f]+1)).
$$
We get
$$
\frac{d}{dt}\mathfrak{E}_\BB[f]\leq \exp(c(l)(\mathfrak{E}_\BB[f]+1)),
$$
and then, we have
\begin{equation}\label{boundIA}
\mathfrak{E}_\BB[f](t)\leq -\frac{1}{c(l)}\log(\exp(-c(l)\mathfrak{E}_\BB[f_0])-c(l)\exp(c(l))t).
\end{equation}
Furthermore, there exists a time $T=T(f_0)$, $T<T^*=\frac{\exp(-c(l)(1+\mathfrak{E}_\BB[f_0]))}{c(l)}$ such that $\mathfrak{E}_{\BB}[f]\leq c(f_0)$. Now, for $\epsilon>0$, we consider
\begin{equation}
\mathcal{J}_\epsilon(x)=\frac{1}{\epsilon}\mathcal{J}\left(\frac{x}{\epsilon}\right),
\label{epsiIA} 
\end{equation}
where $\mathcal{J}$ is the heat kernel, and the regularized problem
\begin{equation}\label{eqregular2IA}
\pat f^{\epsilon,\delta}=F^\epsilon(f^{\epsilon,\delta})\text{ for $x\in\RR$, $t\geq0$ },\qquad f^{\epsilon,\delta}(x,0)=\jeps f_0(x)\text{ for $x\in\RR$},
\end{equation}
where
\begin{multline}\label{eqregularIA}
F^{\epsilon,\delta}(f^{\epsilon,\delta})=\jeps\left(\text{P.V.}\int_\RR \frac{(\jeps\pax f^{\epsilon,\delta}(x)-\jeps\pax f^{\epsilon,\delta}(x-\eta))\sinh(\eta)}{\cosh(\eta)-\cos(\jeps f^{\epsilon,\delta}(x)-\jeps f^{\epsilon,\delta}(x-\eta))+\delta}d\eta\right)\\
+\jeps\left(\text{P.V.}\int_\RR \frac{(\jeps\pax f^{\epsilon,\delta}(x)+\jeps\pax f^{\epsilon,\delta}(x-\eta))\sinh(\eta)}{\cosh(\eta)+\cos(\jeps f^{\epsilon,\delta}(x)+\jeps f^{\epsilon,\delta}(x-\eta))+\delta}d\eta\right).
\end{multline}
For these regularized problems we show the existence of classical solutions $f^{\epsilon,\delta} \in C^1([0,T^\epsilon],H^3(\RR))$. This fact follows from the proof of the local well-posedness result in Sobolev spaces (see Section \ref{sec2.1}). Now we pass to the limit $\delta\rightarrow0$, showing the existence of solutions $f^\epsilon\in C^1([0,T^\epsilon],H^3(\RR))$. Moreover, $f^\epsilon$ are analytic functions. Using the previous \emph{'a priori'} estimates for $\mathfrak{E}_\BB,$ we conclude the existence of an uniform time existence, $T$, for all $f^\epsilon.$ Finally, we can pass to the limit in $\epsilon,$ and we conclude the proof of the smoothing effect.
\end{proof}

\subsection{Well-posedness for analytical initial data}\label{sec2.3}

In this section we show that there is an unique local smooth solution when  the initial data are analytic curves $z(\alpha)$. For a similar result in the case with infinite depth see \cite{ccfgl}. We prove this result by a Cauchy-Kowalevski Theorem (see \cite{ccfgl, nirenberg1972abstract, nishida1977note}). We observe that there is no hypothesis on the Rayleigh-Taylor condition.

The complex extension of the equation can be written as 
$$
\pat z=\frac{\bar{\rho}}{2}F[z],
$$
where
\begin{multline}\label{eqcurvaIA}
F[z]=\text{P.V.}\int_\RR\frac{(\partial_\alpha z(\gamma)-\partial_\alpha z(\gamma-\eta))\sinh(z_1(\gamma)-z_1(\gamma-\eta))}{\cosh(z_1(\gamma)-z_1(\gamma-\eta))-\cos(z_2(\gamma)-z_2(\gamma-\eta))}\\
+\frac{(\partial_\alpha z_1(\gamma)-\partial_\alpha z_1(\gamma-\eta),\partial_\alpha z_2(\gamma)+\partial_\alpha z_2(\gamma-\eta))\sinh(z_1(\gamma)-z_1(\gamma-\eta))}{\cosh(z_1(\gamma)-z_1(\gamma-\eta))+\cos(z_2(\gamma)+z_2(\gamma-\eta))}d\eta.
\end{multline}
We define 
\begin{equation}\label{d-def}
d^-[z](\gamma,\eta)=\frac{\sinh^2(\eta/2)}{\cosh(z_1(\gamma)-z_1(\gamma-\eta))-\cos(z_2(\gamma)-z_2(\gamma-\eta))},
\end{equation}
\begin{equation}\label{d+def}
d^+[z](\gamma,\eta)=\frac{\cosh^2(\eta/2)}{\cosh(z_1(\gamma)-z_1(\gamma-\eta))+\cos(z_2(\gamma)+z_2(\gamma-\eta))}.
\end{equation}
The finiteness of the term $d^ -$ is the arc-chord condition in our domain, \emph{i.e.} if $d^-\in L^\infty$ then $z(\alpha)\neq z(\beta)$ and $|\paa z|>0$. The finiteness of the term $d^+$ means that the curve doesn't touch the boundaries.

We consider curves $z$ in the space 
\begin{equation}\label{Xr}
X_r=\{z \text{ analytic curves on $\BB_r$ such that } \|d^-[z]\|_{L^ \infty}<\infty \text{ and } \makebox{$\|d^+[z]\|_{L^ \infty}<\infty$}\},
\end{equation} 
where \makebox{$\BB_r=\{\alpha+i\xi,\;\;|\xi|<r\}$} and $d^ \pm$ are defined in \eqref{d-def} and \eqref{d+def}. We consider the following norm
\begin{equation}\label{normXr}
\|z\|_r=\|z(\alpha)-(\alpha,0)\|_{H^3(\BB_r)}=\sum_{\pm}\int_\RR|z(\zeta\pm r i)-(\zeta\pm r i,0)|^2d\zeta+\int_\RR|\paa^3 z(\zeta\pm r i)|^2d\zeta,
\end{equation}
where $H^3(\BB_r)$ denotes the Hardy-Sobolev space (see \cite{bakan2007hardy} and references therein). We remark that the fact that they are a Banach scale can be easily proved (see \cite{bakan2007hardy}). For notational convenience we write $\gamma=\alpha\pm ir$, $\gamma'=\alpha\pm ir'$ and we take $\bar{\rho}=2$ and $l=\pi/2$. We claim that, for $0<r'<r$, 
\begin{equation}\label{cauchy2}
\|\partial_\alpha \cdot\|_{L^2(\BB_{r'})}\leq\frac{C}{r-r'}\|\cdot\|_{L^2(\BB_r)}.
\end{equation}
Indeed, we apply Cauchy's integral formula with $\Gamma=\gamma'+(r-r')e^{i\theta}$ to conclude the claim. We need the following result
\begin{prop}\label{curva}
Consider $0\leq r'<r$ and the set
$$
O_R=\{z\in X_r\text{ such that }\|z\|_r<R, \|d^-[z]\|_{L^\infty(\BB_r)}<R, \|d^+[z]\|_{L^\infty(\BB_r)}<R\},  
$$
where $d^-[z]$ and $d^+[z]$ are defined in \eqref{d-def} and \eqref{d+def}. Then $F:O_R\rightarrow X_{r'}$ is continuous. Moreover, for $z,w\in O_R$, the following inequalities holds:
\begin{eqnarray}
&&\|F[z]\|_{H^3(\BB_{r'})}\leq\frac{C_R}{r-r'}\|z\|_r,\label{7.1}\\
&&\|F[z]-F[w]\|_{H^3(\BB_{r'})}\leq\frac{C_R}{r-r'}\|z-w\|_{H^3(\BB_r)},\label{7.2}\\
&&\sup_{\gamma\in \BB_r,\beta\in\RR}|F[z](\gamma)-F[z](\gamma-\beta)|\leq C_R|\beta|.\label{7.3}
\end{eqnarray}
\end{prop}
The proof of this proposition follows the same ideas as in \cite{ccfgl} and we left for the interested reader. With this Proposition we can prove the local existence result:

\begin{thm}\label{curvaCK2}
Let $z_0\in X_{r_0}$, for some $r_0>0$ (see definition \eqref{Xr}), be the initial data and assume that this initial data does not reach the boundaries and that the arc-chord condition is satisfied. Then there exists an analytic solution of the Muskat problem \eqref{eqcurva} for $t\in[-T,T]$ for a small enough $T>0$. 
\end{thm}

\begin{proof}
Notice that $z_0\in X_{r_0}$ satisfies the arc-chord condition and does not reach the boundaries. Then, there exists $R_0$ such that $z_0\in O_{R_0}$. We take $r<r_0$ and $R>R_0$ in order to define $O_R$ and we consider the iterates
$$
z_{n+1}=z_0+\int_0^tF[z_n]ds,
$$
and assume by induction that $z_k\in O_R$ for $k\leq n$. Then, following the proofs in \cite{ccfgl, nirenberg1972abstract, nishida1977note}, we obtain a time $T_{CK}>0$. It remains to show that 
$$
\|d^-[z_{n+1}]\|_{L^\infty(\BB_r)}<R,\qquad\|d^+[z_{n+1}]\|_{L^\infty(\BB_r)}<R,
$$
for some times $T_A, T_B>0$ respectively. We have
\begin{eqnarray*}
(d^-[z_{n+1}])^{-1}&=&\frac{\cosh\left(z_{0_1}(\gamma)-z_{0_1}(\gamma-\eta)+\int_0^t F^1[z](\gamma)-F^1[z](\gamma-\eta)ds\right)}{\sinh^2(\eta/2)}\\
&&-\frac{\cos\left(z_{0_2}(\gamma)-z_{0_2}(\gamma-\eta)+\int_0^t F^2[z](\gamma)-F^2[z](\gamma-\eta)ds\right)}{\sinh^2(\eta/2)}\\
&=&\frac{\cosh\left(z_{0_1}(\gamma)-z_{0_1}(\gamma-\eta)\right)\cosh\left(\int_0^t F^1[z](\gamma)-F^1[z](\gamma-\eta)ds\right)}{\sinh^2(\eta/2)}\\
&&+\frac{\sinh\left(z_{0_1}(\gamma)-z_{0_1}(\gamma-\eta)\right)\sinh\left(\int_0^t F^1[z](\gamma)-F^1[z](\gamma-\eta)ds\right)}{\sinh^2(\eta/2)}\\
&&-\frac{\cos\left(z_{0_2}(\gamma)-z_{0_2}(\gamma-\eta)\right)\cos\left(\int_0^t F^2[z](\gamma)-F^2[z](\gamma-\eta)ds\right)}{\sinh^2(\eta/2)}\\
&&+\frac{\sin\left(z_{0_2}(\gamma)-z_{0_2}(\gamma-\eta)\right)\sin\left(\int_0^t F^2[z](\gamma)-F^2[z](\gamma-\eta)ds\right)}{\sinh^2(\eta/2)}.
\end{eqnarray*}
Using the classical trigonometric formulas we obtain
\begin{eqnarray*}
(d^-[z_{n+1}])^{-1}
&=&(d^-[z_0])^{-1}+\frac{\cosh\left(z_{0_1}(\gamma)-z_{0_1}(\gamma-\eta)\right)2\sinh^2\left((\int_0^t F^1[z](\gamma)-F^1[z](\gamma-\eta)ds)/2\right)}{\sinh^2(\eta/2)}\\
&&+\frac{\sinh\left(z_{0_1}(\gamma)-z_{0_1}(\gamma-\eta)\right)\sinh\left(\int_0^t F^1[z](\gamma)-F^1[z](\gamma-\eta)ds\right)}{\sinh^2(\eta/2)}\\
&&-\frac{\cos\left(z_{0_2}(\gamma)-z_{0_2}(\gamma-\eta)\right)2\sin^2\left((\int_0^t F^2[z](\gamma)-F^2[z](\gamma-\eta)ds)/2\right)}{\sinh^2(\eta/2)}\\
&&+\frac{\sin\left(z_{0_2}(\gamma)-z_{0_2}(\gamma-\eta)\right)\sin\left(\int_0^t F^2[z](\gamma)-F^2[z](\gamma-\eta)ds\right)}{\sinh^2(\eta/2)}\\
\end{eqnarray*}
Take $t\leq1$, assuming that $\eta\in B(0,1)$ and using the inequality \eqref{7.3} in Proposition \ref{curva}, we have
$$
(d^-[z_{n+1}])^{-1}> \frac{1}{R_0}-C^1_R(t^2+t).
$$
In the case where $\eta\in B^c(0,1)$, to ensure the decay at infinity, we use the inequality \eqref{7.1} in Proposition \ref{curva} to get
$$
(d^-[z_{n+1}])^{-1}>\frac{1}{R_0}-C^2_R(t^2+t).
$$
Thus, we can take 
$$
0<T_A<\min\left\{1,\sqrt{\left(\frac{1}{R_0}-\frac{1}{R}\right)\frac{1}{4\max\{C^1_R,C^2_R\}}}\right\},
$$
and then $\|d^-[z_{n+1}]\|_{L^\infty(\BB_r)}<R$. We proceed in the same way for $d^+[z_{n+1}].$ Using the classical trigonometric formulas and the previous inequalities we obtain
\begin{eqnarray*}
(d^+[z_{n+1}])^{-1}
&=&(d^+[z_0])^{-1}+\frac{\cosh\left(z_{0_1}(\gamma)-z_{0_1}(\gamma-\eta)\right)2\sinh^2\left((\int_0^t F^1[z](\gamma)-F^1[z](\gamma-\eta)ds)/2\right)}{\cosh^2(\eta/2)}\\
&&+\frac{\sinh\left(z_{0_1}(\gamma)-z_{0_1}(\gamma-\eta)\right)\sinh\left(\int_0^t F^1[z](\gamma)-F^1[z](\gamma-\eta)ds\right)}{\cosh^2(\eta/2)}\\
&&-\frac{\cos\left(z_{0_2}(\gamma)+z_{0_2}(\gamma-\eta)\right)2\sin^2\left((\int_0^t F^2[z](\gamma)+F^2[z](\gamma-\eta)ds)/2\right)}{\cosh^2(\eta/2)}\\
&&+\frac{\sin\left(z_{0_2}(\gamma)+z_{0_2}(\gamma-\eta)\right)\sin\left(\int_0^t F^2[z](\gamma)+F^2[z](\gamma-\eta)ds\right)}{\cosh^2(\eta/2)}\\
\end{eqnarray*}
thus, we can consider 
$$
0<T_B<\min\left\{1,\sqrt{\left(\frac{1}{R_0}-\frac{1}{R}\right)\frac{1}{2C^3_R}}\right\},
$$
and then $\|d^+[z_{n+1}]\|_{L^\infty(\BB_r)}<R$. Taking $0<T<\min\{T_{CK},T_A,T_B\}$, we conclude the proof. 
\end{proof}

\begin{coment}
We remark that there is not any hypothesis on the Rayleigh-Taylor condition in this existence result. As a corollary, we obtain that if we start in the Rayleigh-Taylor unstable case with an analytic graph, we have local existence and uniqueness for a short time. The same result can be proven in the more general case of $C^{2,\delta}(\BB_r)$ functions.
\end{coment}

\subsection{Ill-posedness in the Rayleigh-Taylor unstable case}\label{sec2.4}
Now, if we consider $\rho^2<\rho^1$ in our system, the problem is ill-posed in Sobolev spaces, \emph{i.e.} a singularity appears for arbitrarly small initial energies and times. We remark that, if the initial data is analytic and does not reach initially the boundaries, there is local existence (see Section \ref{sec2.3}). The idea is to use the instant analyticity forward in time to conclude our result. 
\begin{thm}\label{teoil}
There exists a solution $\tilde{f}$ of \eqref{eq0.1} with $\rho^2<\rho^1$ such that $\|\tilde{f}_0\|_{H^s(\RR)}<\epsilon$ and $\|\tilde{f}(\delta)\|_{H^s(\RR)}=\infty$, for any $s\geq4$, $\epsilon>0$ and small enough $\delta>0$.
\end{thm}
\begin{proof} We prove the case $s=4,$ being analogous the rest of the cases. Take $g_0(x)\in H^3(\RR)$ but $g_0\notin H^4(\RR)$. We consider a fixed constant $R\geq4$ and $0<\lambda<1$. Now we denote $f^\lambda(x,t)$ the solution to the problem \eqref{eq0.1} with initial datum $f^\lambda (x,0)=\lambda g_0(x)$. We know that $f^\lambda$ exists for a positive time $T(\lambda,g_0)$ and that it is analytic in a complex strip which grows with constant $k(\lambda,g_0)$ (see Sections \ref{sec2.1}, \ref{sec2.2} and equation \eqref{kcond}). We can take an uniform $k^*$ with respect to $\lambda$. Indeed, using the definition of \eqref{m},
$$
k(\lambda,g_0)=\frac{\pi m(0)}{4}=\frac{\pi}{4}\frac{1}{1+\lambda^2\|\pax g_0\|^2_{L^\infty(\RR)}}\geq \frac{\pi}{4}\frac{1}{1+\|\pax g_0\|^2_{L^\infty(\RR)}}=k^*(g_0).
$$
Then the condition
$$
4k^*(g_0)-2\pi m(0)=4k^*(g_0)-\frac{2\pi}{1+\lambda^2\|\pax g_0\|_{L^\infty(\RR)}^2}<0 
$$
is satisfied. We have
\begin{multline*}
\mathfrak{E}_{\BB}[f^\lambda (0)]=\lambda^2\|g_0\|^2_{H^3(\RR)}+\|d^-[\lambda g_0]\|_{L^\infty(\RR)}+\|d^+[\lambda g_0]\|_{L^\infty(\RR)}\\
+\|D[\lambda g_0]\|_{L^\infty(\RR)}+\frac{1+\lambda^2\|\pax g_0\|^2_{L^\infty(\RR)}}{\pi},
\end{multline*}
with
$$
\|d^- [\lambda g_0]\|_{L^\infty(\RR)}=\frac{\sinh^2(\eta/3)}{2\sinh^2(\eta/2)\left(1+\frac{\sin^2\left(\frac{\lambda}{2}(g_0(x)-g_0(x-\eta))\right)}{\sinh^2(\eta/2)}\right)}\leq \frac{2}{9},
$$
\begin{multline*}
\|d^+ [\lambda g_0]\|_{L^\infty(\RR)}=\frac{\cosh^2(\eta/3)}{2\cosh^2(\eta/2)\left(1-\frac{\sin^2\left(\frac{\lambda}{2}(g_0(x)+g_0(x-\eta))\right)}{\cosh^2(\eta/2)}\right)}\\\leq\frac{\cosh^2(\eta/3)}{2\cosh^2(\eta/2)\left(1-\frac{\sin^2\left(\frac{1}{2}(g_0(x)+g_0(x-\eta))\right)}{\cosh^2(\eta/2)}\right)}=\|d^+ [g_0]\|_{L^\infty(\RR)},
\end{multline*}
and
$$
\|D [\lambda g_0]\|_{L^\infty(\RR)}=\frac{1}{\cosh(R)-2\cosh(2\lambda\|g_0\|_{L^\infty(\RR)})}\leq\frac{1}{\cosh(4)-2\cosh(\pi)}.
$$
Therefore, we conclude the following uniform bound
$$
\mathfrak{E}_{\BB}[f^\lambda (0)]\leq c(g_0),
$$
with $c(g_0)$ some constant depending on $g_0$ that changes from line to line. We also take 
$$
0<\delta^*(g_0)=\frac{1}{c\exp(c(g_0))}\leq\min_{\lambda\in[0,1]}T(\lambda,g_0).
$$ 
Now, we consider $0<\delta<\delta^*(g_0)$. We remark that all $f^\lambda (x,t)$ exists up to time $\delta^*(g_0)$ and, by means of the instant analyticity result in Section \ref{sec2.3}, we have 
\begin{equation}\label{unilambda}
\mathfrak{E}_{\BB}[f^\lambda(t)]\leq c(g_0),\;\;\forall 0<t<\delta^*(g_0).
\end{equation}
Now define $\tilde{f}^{\lambda,\delta}(x,t)=f(x,-t+\delta)$. We have 
$$
\|\tilde{f}^{\lambda,\delta}(\delta)\|_{H^4(\RR)}=\lambda\|g_0\|_{H^4(\RR)}=\infty.
$$ 
Recall that $f^\lambda$ is analytic in the common complex strip growing with constant $k^*(g_0)$ for all $0<\lambda<1$. Then, applying Cauchy's integral formula with the curve $\Gamma=x+k^*(g_0)\delta e^{i\theta}$ and that Hardy spaces on growing strips are a Banach scale, we get
$$
\|\pax^4 \tilde{f}^{\lambda,\delta}(0)\|_{L^2(\RR)}=\|\pax^4 f^{\lambda}(\delta)\|_{L^2(\RR)}\leq \frac{C}{k^*(g_0)\delta}\|\pax^3 f^\lambda\|_{L^2(\BB_{\delta^*})}.
$$
Using the uniform energy bound \eqref{unilambda} we have 
$$
\|\pax^3 f^\lambda\|_{L^2(\BB_\delta^*)}\leq c(g_0)\lambda\|\pax^3 g_0\|_{L^2(\RR)},
$$
thus
$$
\|\pax^4 \tilde{f}^{\lambda,\delta}(0)\|_{L^2(\RR)}\leq \frac{c(g_0)}{\delta}\lambda.
$$
Now, given $\epsilon>0$ take $0<\lambda=\min\left\{1,\frac{\delta\epsilon}{c(g_0)}\right\}$ to conclude $\|\pax^4 \tilde{f}^{\lambda,\delta}(0)\|_{L^2(\RR)}<\epsilon.\quad$\end{proof}

\begin{coment}We remark that the problem in the whole plane is also ill-posed in the unstable regime case but the proof of this fact is different and depends on the kernel appearing in the whole plane. Therefore, the same idea in the proof can not work in the confined case. Moreover, the ill-posedness result above is different from those in \cite{c-g07, SCH} because we do not require a family of solutions having arbitrary long time existence.
\end{coment}


\section{Differences between the two regimes}\label{sec3}

In this section we study some properties  of the regime with $0<\mathcal{A}<1$ (finite depth) which are different with respect to the regime $\mathcal{A}=0$ (infinite depth).

We show the qualitative behaviour for $\|f\|_{L^ \infty}$ and $\|\pax f\|_{L^ \infty}$, and the existence of turning waves. A \emph{'turning wave'} is a blow up for $\pax f$ (see \cite{ccfgl}). 

\subsection{Maximum Principles}\label{sec3.1}

In this section we show a maximum principle and a decay estimate for $\|f\|_{L^\infty}$ and a maximum principle for $\|\pax f\|_{L^\infty}$ for a special class of initial data. In order to 
prove the maximum principles, the key point is to compare the local and the nonlocal terms that appear in the ODEs for the evolution of the $L^\infty$ norms.

\subsubsection{Maximum principle for $\|f\|_{L^\infty}$}
\begin{thm}\label{teo2}
Let $f(t)\in H^ 3_l(\RR)$ be the unique classical solution of \eqref{eq0.1} in the Rayleigh-Taylor stable case. Then $f$ satisfies that 
$$
\|f(t)\|_{L^\infty}\leq \|f_0\|_{L^\infty}.
$$
\end{thm}
\begin{proof} Due to the smoothness of $f$ in space and time we have that $\|f(t)\|_{L^\infty}=f(x_t)$ is Lipschitz. Then, using Rademacher Theorem, we have that $f(x_t)$ is differentiable almost everywhere and thus we get that 
\begin{equation}\label{eq-MP}
\frac{d}{dt}\|f(t)\|_{L^\infty}=\pat f(x_t)=\text{P.V.}\int_\RR\partial_\eta f(x_t-\eta)(\Xi_1(x_t,\eta,f)-\Xi_2(x_t,\eta,f))d\eta=I_1+I_2.
\end{equation}
Let us introduce the following notation:
\begin{equation}\label{theta}
\theta=\frac{f(x_t)-f(x_t-\eta)}{2},
\qquad
\bar{\theta}=\frac{f(x_t)+f(x_t-\eta)}{2}.
\end{equation}
Now, since $\Xi_1$ is defined as \eqref{Xi1}, using the classical and hyperbolic trigonometric formulas for the half-angle, we have
\begin{eqnarray*}
I_1&=&-2\text{P.V.}\int_\RR\partial_\eta \theta\frac{2\sinh\left(\frac{\eta}{2}\right)\cosh\left(\frac{\eta}{2}\right)}{\cosh\left(\eta\right)-\frac{1-\tan^2\left(\theta\right)}{1+\tan^2\left(\theta\right)}}=-2\text{P.V.}\int_\RR\partial_\eta \theta\frac{2\sinh\left(\frac{\eta}{2}\right)\cosh\left(\frac{\eta}{2}\right)}{\cosh\left(\eta\right)-1+\frac{2\tan^2\left(\theta\right)}{1+\tan^2\left(\theta\right)}}\\
&=&-2\text{P.V.}\int_\RR\partial_\eta \theta\frac{(1+\tan^2(\theta))\coth\left(\frac{\eta}{2}\right)}{1+\tan^2\left(\theta\right)\coth^2\left(\frac{\eta}{2}\right)}=-2\text{P.V.}\int_\RR\frac{\partial_\eta \theta}{\cos^2(\theta)}\frac{\coth\left(\frac{\eta}{2}\right)}{1+\tan^2\left(\theta\right)\coth^2\left(\frac{\eta}{2}\right)}\\
&=&-\text{P.V.}\int_\RR\partial_\eta\tan(\theta)\frac{2\coth\left(\frac{\eta}{2}\right)}{1+(\tan(\theta)\coth(\eta/2))^2}d\eta\\
&=&-4\frac{\tan\left(\frac{f(x_t)}{2}\right)}{1+\tan^2\left(\frac{f(x_t)}{2}\right)}+\text{P.V.}\int_\RR\tan(\theta)\partial_\eta\left(\frac{2\coth\left(\frac{\eta}{2}\right)}{1+(\tan(\theta)\coth(\eta/2))^2}\right)d\eta,
\end{eqnarray*}
where we integrate by parts. Considering
$$
G(x)=\frac{-x}{1+x^2}+\arctan(x),
$$
we have
\begin{multline*}
I_1 =-4\frac{\tan\left(\frac{f(x_t)}{2}\right)}{1+\tan^2\left(\frac{f(x_t)}{2}\right)}-2\int_\RR\partial_\eta \left[G\left(\frac{\tan(\theta)}{\tanh(\frac\eta 2)}\right)\right]d\eta\\-\text{P.V.}\int_\RR\frac{\tan(\theta)}{\sinh^2(\frac\eta 2)}\frac{1}{1+(\tan(\theta)\coth(\frac\eta 2))^2}d\eta.
\end{multline*}
Then, we get
\begin{equation}
I_1=-2f(x_t)-\text{P.V.}\int_\RR\frac{\tan(\theta)}{\sinh^2(\eta/2)}\frac{1}{1+(\tan(\theta)\coth(\eta/2))^2}d\eta.
\label{max1}
\end{equation}
Anagously for the $I_2$ term, we use the classical and hyperbolic trigonometric formulas. In this case we have to write all in terms of the $\tan(\bar{\theta})$. This is possible because $x_t$ is a maximum point. Since $\Xi_2$ is defined as \eqref{Xi2} and using the same function $G$ evaluated in $\tan(\bar{\theta})\tanh(\eta/2)$, we get the following expression for $I_2$ 
\begin{eqnarray}
I_2&=&-2\text{P.V.}\int_\RR\partial_\eta \bar{\theta}\frac{2\sinh\left(\frac{\eta}{2}\right)\cosh\left(\frac{\eta}{2}\right)}{\cosh\left(\eta\right)+\frac{1-\tan^2\left(\bar{\theta}\right)}{1+\tan^2\left(\theta\right)}}=-2\text{P.V.}\int_\RR\partial_\eta \bar{\theta}\frac{2\sinh\left(\frac{\eta}{2}\right)\cosh\left(\frac{\eta}{2}\right)}{\cosh\left(\eta\right)-1+\frac{2}{1+\tan^2\left(\bar{\theta}\right)}}\nonumber\\
&=&-2\text{P.V.}\int_\RR\partial_\eta \bar{\theta}\frac{(1+\cot^2(\bar{\theta}))\coth\left(\frac{\eta}{2}\right)}{1+\cot^2\left(\bar{\theta}\right)\coth^2\left(\frac{\eta}{2}\right)}=-2\text{P.V.}\int_\RR\frac{\partial_\eta \bar{\theta}}{\sin^2(\bar{\theta})}\frac{\coth\left(\frac{\eta}{2}\right)}{1+\cot^2\left(\bar{\theta}\right)\coth^2\left(\frac{\eta}{2}\right)} \nonumber\\
&=&-2\text{P.V.}\int_\RR\frac{\partial_\eta \bar{\theta}}{\cos^2(\bar{\theta})}\frac{\tanh\left(\frac{\eta}{2}\right)}{1+\tan^2\left(\bar{\theta}\right)\tanh^2\left(\frac{\eta}{2}\right)}=-2\text{P.V.}\int_\RR\partial_\eta \tan(\bar{\theta})\frac{\tanh\left(\frac{\eta}{2}\right)}{1+\tan^2\left(\bar{\theta}\right)\tanh^2\left(\frac{\eta}{2}\right)} \nonumber\\ \nonumber\\
&=&-2f(x_t)+\text{P.V.}\int_\RR\frac{\cot(\bar{\theta})}{\sinh^2(\eta/2)}\frac{1}{1+(\cot(\bar{\theta})\coth(\eta/2))^2}d\eta.
\label{max2}
\end{eqnarray}
Therefore, by \eqref{max1} and \eqref{max2}, we have in \eqref{eq-MP}
\begin{multline}\label{eqM}
\partial_tf(x_t)=-4f(x_t)+\int_\RR\frac{\cot(\bar{\theta})}{\cosh^2(\eta/2)}\frac{1}{\tanh^2(\eta/2)+\cot^2(\bar{\theta})}d\eta\\
-\text{P.V.}\int_\RR\frac{\tan(\theta)}{\cosh^2(\eta/2)}\frac{1}{\tanh^2(\eta/2)+\tan^2(\theta)}d\eta.
\end{multline}
Due to the definition of $\bar\theta$
$$
\cot(\bar{\theta})=\tan\left(\frac{\pi}{2}-\bar{\theta}\right)=\tan\left(\frac{\pi}{2}-f(x_t)+\theta\right).
$$
Moreover, using that
$$
\arctan(\tan(f(x_t))\tanh(\eta/2))\bigg{|}^\infty_{-\infty}=2f(x_t),
$$
we can write 
$$
4f(x_t)=\int_\RR\frac{1}{\cosh^2\left(\frac{\eta}{2}\right)}\frac{\tan(\frac{\pi}{2}-f(x_t))}{\tan^2(\frac{\pi}{2}-f(x_t))+\tanh^2\left(\frac{\eta}{2}\right)}d\eta,
$$
and we use the equality
$$
\tan\left(\frac{\pi}{2}-f(x_t)+\theta\right)=\frac{\tan(\frac{\pi}{2}-f(x_t))+\tan(\theta)}{1-\tan(\frac{\pi}{2}-f(x_t))\tan(\theta)}.
$$
By notational convenience we write $\sigma=\frac{\pi}{2}-f(x_t)$. We define
\begin{multline*}
\Pi(x,\eta,t)= \frac{\tan(\sigma)}{\tan^2(\sigma)+\tanh^2\left(\frac{\eta}{2}\right)}+\frac{\tan(\theta)}{\tan^2(\theta)+\tanh^2\left(\frac{\eta}{2}\right)}\\
 -\frac{(\tan(\sigma)+\tan(\theta))(1-\tan(\sigma)\tan(\theta))}{(\tan(\sigma)+\tan(\theta))^2+(1-\tan(\sigma)\tan(\theta))^2\tanh^2\left(\frac{\eta}{2}\right)},
\end{multline*}
and, using \eqref{eqM}, we have
\begin{equation}\label{PII}
\partial_tf(x_t)=-\int_\RR\frac{1}{\cosh^2(\eta/2)}\Pi(x,\eta,t)d\eta.
\end{equation}
So we need to prove that $\Pi\geq0$ (respectively $\leq0$) if $\|f\|_{L^\infty}=\max_x f(x)$ (respectively $\min_x f(x)$). We have
\begin{multline}\label{PIII}
\Pi=\frac{\tan(\sigma)\tan(\theta)(\tan(\theta)+2\tan(\sigma))}{[\tan^2(\sigma)+\tanh^2\left(\frac{\eta}{2}\right)][(\tan(\sigma)+\tan(\theta))^2+(1-\tan(\sigma)\tan(\theta))^2\tanh^2\left(\frac{\eta}{2}\right)]}\\
+\frac{\tan^2(\sigma)\tan(\theta)(\tan(\sigma)\tan(\theta)-2)\tanh^2\left(\frac{\eta}{2}\right)}{[\tan^2(\sigma)+\tanh^2\left(\frac{\eta}{2}\right)][(\tan(\sigma)+\tan(\theta))^2+(1-\tan(\sigma)\tan(\theta))^2\tanh^2\left(\frac{\eta}{2}\right)]}\\
+\frac{\tan(\sigma)\tan(\theta)(\tan(\sigma)+2\tan(\theta))}{[\tan^2(\theta)+\tanh^2\left(\frac{\eta}{2}\right)][(\tan(\sigma)+\tan(\theta))^2+(1-\tan(\sigma)\tan(\theta))^2\tanh^2\left(\frac{\eta}{2}\right)]}\\
+\frac{\tan(\sigma)\tan^2(\theta)(\tan(\sigma)\tan(\theta)-2)\tanh^2\left(\frac{\eta}{2}\right)}{[\tan^2(\theta)+\tanh^2\left(\frac{\eta}{2}\right)][(\tan(\sigma)+\tan(\theta))^2+(1-\tan(\sigma)\tan(\theta))^2\tanh^2\left(\frac{\eta}{2}\right)]}\\
+\frac{(\tan(\sigma)+\tan(\theta))\tan(\sigma)\tan(\theta)}{(\tan(\sigma)+\tan(\theta))^2+(1-\tan(\sigma)\tan(\theta))^2\tanh^2\left(\frac{\eta}{2}\right)}.
\end{multline}
Rearranging, we get
\begin{multline*}
 \Pi=\frac{\tan(\sigma)\tan^2(\theta)[1+\tan^2(\sigma)\tanh^2\left(\frac{\eta}{2}\right)]}{[\tan^2(\sigma)+\tanh^2\left(\frac{\eta}{2}\right)][(\tan(\sigma)+\tan(\theta))^2+(1-\tan(\sigma)\tan(\theta))^2\tanh^2\left(\frac{\eta}{2}\right)]}\\
+\frac{2\tan^2(\sigma)\tan(\theta)[1-\tanh^2\left(\frac{\eta}{2}\right)]}{[\tan^2(\sigma)+\tanh^2\left(\frac{\eta}{2}\right)][(\tan(\sigma)+\tan(\theta))^2+(1-\tan(\sigma)\tan(\theta))^2\tanh^2\left(\frac{\eta}{2}\right)]}\\
+\frac{\tan^2(\sigma)\tan(\theta)[1+\tan^2(\theta)\tanh^2\left(\frac{\eta}{2}\right)]}{[\tan^2(\theta)+\tanh^2\left(\frac{\eta}{2}\right)][(\tan(\sigma)+\tan(\theta))^2+(1-\tan(\sigma)\tan(\theta))^2\tanh^2\left(\frac{\eta}{2}\right)]}\\
+\frac{2\tan(\sigma)\tan^2(\theta)[1-\tanh^2\left(\frac{\eta}{2}\right)]}{[\tan^2(\theta)+\tanh^2\left(\frac{\eta}{2}\right)][(\tan(\sigma)+\tan(\theta))^2+(1-\tan(\sigma)\tan(\theta))^2\tanh^2\left(\frac{\eta}{2}\right)]}\\
+\frac{(\tan(\sigma)+\tan(\theta))\tan(\sigma)\tan(\theta)}{(\tan(\sigma)+\tan(\theta))^2+(1-\tan(\sigma)\tan(\theta))^2\tanh^2\left(\frac{\eta}{2}\right)}.
\end{multline*}
Now, we use that $\tanh^2\left(\frac{\eta}{2}\right)\leq1$. If $\|f\|_{L^\infty}=\max_x f(x)$, the definitions of $\theta$ and $\sigma$ give us that $\tan(\theta),\tan(\sigma)>0$  and obtaining that $\Pi\geq 0$. This concludes the proof for this case.
For the case where the $L^\infty$ norm is achieved in the minimum the proof is analogous. Indeed, we have that in this case $\Pi\leq0$ because $\tan(\sigma),\tan(\theta)\leq0$. \end{proof}

\begin{coment}
The main difference of this result with respect to the one in \cite{c-g09} is that we have positive and negative terms in \eqref{eqM}. Thus we have to balance them to obtain our result. We note that the local terms dissapear with infinite depth.
\end{coment}

\subsubsection{A decay estimate for $\|f(t)\|_{L^\infty}$}

\begin{thm}\label{decay}
Let $f_0\geq0$, $f_0\in L^1\cap H^k_l(\RR)$ be the initial datum and assume $\rho^2-\rho^1>0$. Then the solution $f(x,t)$ of equation \eqref{eq0.1} satisfies the inequality 
$$
\frac{d}{dt}\|f(t)\|_{L^\infty}\leq-c(\|f_0\|_{L^1},\|f_0\|_{L^\infty},\rho^2,\rho^1,l)e^{-\frac{\pi\|f_0\|_{L^1}}{l\|f(t)\|_{L^\infty}}}.
$$
\end{thm}
\begin{proof} We conserve the notation and the hypothesis of the proof of Theorem~\ref{teo2}, $i.e.$ $f(x_t)=\|f(t)\|_{L^\infty}$, $\sigma=\frac{\pi}{2}-f(x_t)$ and $\theta=\frac{f(x_t)-f(x_t-\eta)}{2}$. We have the equation \eqref{PII} with $\Pi$ defined in \eqref{PIII}. Due to analysis in the proof of Theorem \ref{teo2}, we have the following bound
\begin{eqnarray*}
\Pi&\geq&\frac{\tan^2(\sigma)\tan(\theta)[1+\tan^2(\theta)\tanh^2\left(\frac{\eta}{2}\right)]}{[\tan^2(\theta)+\tanh^2\left(\frac{\eta}{2}\right)][(\tan(\sigma)+\tan(\theta))^2+(1-\tan(\sigma)\tan(\theta))^2\tanh^2\left(\frac{\eta}{2}\right)]}\\
&\geq&\frac{\tan(\theta)}{(\tan^2(\|f_0\|_{L^\infty})+1)^2+\tan^2(\|f_0\|_{L^\infty})}\frac{1}{1+\tan^2(\|f_0\|_{L^\infty})}.\\
\end{eqnarray*}
thus
\begin{equation}\label{inq-MP}
-\partial_tf(x_t)\geq\int_{\RR}\frac{\cosh^{-2}(\eta/2)\tan(\theta)}{((\tan^2(\|f_0\|_{L^\infty})+1)^2+\tan^2(\|f_0\|_{L^\infty}))(1+\tan^2(\|f_0\|_{L^\infty}))}d\eta.
\end{equation}
Let $r>0$ be a parameter that will be chosen below and define the interval $[-r,r]$. We consider the sets
$$
\mathcal{U}_1=\left\{\eta:\eta\in[-r,r], \theta\geq\frac{f(x_t)}{4}\right\}
$$
and
$$
\mathcal{U}_2=\left\{\eta:\eta\in[-r,r], \theta<\frac{f(x_t)}{4}\right\}.
$$
We observe that $\mathcal{U}_2$ is not empty for every $r>0$. The conservation of the total mass \eqref{mean} gives us a control for the measure of these sets. Indeed,
$$
\|f_0\|_{L^1}=\int_\RR f(x_t-\eta)d\eta\geq\int_{\mathcal{U}_2} f(x_t-\eta)d\eta>\frac{f(x_t)}{2}|\mathcal{U}_2|.
$$
Therefore 
$$|\mathcal{U}_1|=2r-|\mathcal{U}_2|\geq 2\left(r-\frac{\|f_0\|_{L^1}}{f(x_t)}\right).$$
Notice that if $r>\frac{\|f_0\|_{L^1}}{f(x_t)}$ we obtain that $|\mathcal{U}_1|>0$. Using \eqref{inq-MP}, we have
\begin{eqnarray*}
-\partial_tf(x_t)&\geq&\int_{\mathcal{U}_1}\frac{1}{\cosh^2(\eta/2)}\frac{\tan(\theta)}{((\tan^2(\|f_0\|_{L^\infty})+1)^2+\tan^2(\|f_0\|_{L^\infty}))(1+\tan^2(\|f_0\|_{L^\infty}))}d\eta\\
&\geq&\frac{1}{\cosh^2(r/2)}\frac{\tan(f(x_t)/4)}{((\tan^2(\|f_0\|_{L^\infty})+1)^2+\tan^2(\|f_0\|_{L^\infty}))(1+\tan^2(\|f_0\|_{L^\infty}))}|\mathcal{U}_1|\\
&\geq&\frac{rf(x_t)-\|f_0\|_{L^1}}{2\cosh^2(r/2)((\tan^2(\|f_0\|_{L^\infty})+1)^2+\tan^2(\|f_0\|_{L^\infty}))(1+\tan^2(\|f_0\|_{L^\infty}))}.
\end{eqnarray*}
Now, we fix
$$
r=2\frac{\|f_0\|_{L^1}}{f(x_t)}.
$$
Then, we get
\begin{multline*}
-\partial_tf(x_t)\geq\frac{\cosh^{-2}\left(\frac{\|f_0\|_{L^1}}{f(x_t)}\right)\|f_0\|_{L^1}}{2((\tan^2(\|f_0\|_{L^\infty})+1)^2+\tan^2(\|f_0\|_{L^\infty}))(1+\tan^2(\|f_0\|_{L^\infty}))}\\
\geq\frac{e^{-\frac{2\|f_0\|_{L^1}}{f(x_t)}}\|f_0\|_{L^1}}{2((\tan^2(\|f_0\|_{L^\infty})+1)^2+\tan^2(\|f_0\|_{L^\infty}))(1+\tan^2(\|f_0\|_{L^\infty}))}, 
\end{multline*}
and we conclude the proof. \end{proof}

\begin{coment}
We observe that in the whole plane case (see \cite{c-g09}) the decay rate is given by
$$
\partial_tf(x_t)\leq -c(\|f_0\|_{L^1},\|f_0\|_{L^\infty},\rho^1,\rho^2)(f(x_t))^2,
$$
so in the case without boundaries the decay is faster.
\end{coment}

\begin{coment}
As a corollary we conclude that there are no one-signed, integrable, steady state solutions. 
\end{coment}

\subsubsection{Maximum principle for $\|\pax f(t)\|_{L^\infty}$}

In this section we show the maximum principle for $\|\pax f (t)\|_{L^\infty}$ for a special class of initial data:

\begin{thm}\label{decayderiv}
Let $f_0\in H^3_l(\RR)$ be a smooth initial datum such that conditions \eqref{H3},\eqref{H4} and \eqref{H5} hold. Then, the solution $f(x,t)$ of equation \eqref{eq0.1} satisfies
\begin{equation}\label{MPD}
\|\pax f(t)\|_{L^\infty}\leq \|\pax f_0\|_{L^\infty}.
\end{equation}
Moreover, if $(x(l),y(l))$ is the solution of the system \eqref{sisder} and assuming that   $\|f_0\|_{L^\infty}<x(l)$ and $\|\pax f_0\|_{L^\infty}<y(l)$, we have that 
$$
\|\pax f (t)\|_{L^\infty}\leq 1.
$$
\end{thm}

\begin{proof}Using the same method as in Theorem \ref{teo2} and the smoothness of $f$, we have that the evolution of $\pax f(x_t)=\|\pax f (t)\|_{L^\infty}$ is given by 
$$
\frac{d}{dt}\|\pax f\|_{L^\infty}=\pat \pax f(x_t).
$$
We suppose that $\pax f(x_t)=\max \pax f(x,t)$ in order to clarify the exposition, but the proof is analogous in the case where the norm is achieved by the minimum. Since \eqref{eq0.1} is equivalent to \eqref{eqderiv}, so we have to take a derivative in space in this equivalent formulation. The boundaries in the principal value integrals contributes with $-8\pax f(x_t)$. Thus we get
$$
\pat \pax f(x_t)=-8\pax f(x_t)+I_1+I_2,
$$
with
$$
I_1=2\text{P.V.}\int_\RR\pax^2\left(\arctan\left(\frac{\tan\left(\frac{f(x_t)-f(\eta)}{2}\right)}{\tanh\left(\frac{x_t-\eta}{2}\right)}\right)\right)d\eta,
$$
and
$$
I_2=2\text{P.V.}\int_\RR\partial_x^2\left(\arctan\left(\tan\left(\frac{f(x_t)+f(\eta)}{2}\right)\tanh\left(\frac{x_t-\eta}{2}\right)\right)\right)d\eta.
$$
We define
$$
\mu_1=\frac{\tan\left(\frac{f(x_t)-f(\eta)}{2}\right)}{\tanh\left(\frac{x_t-\eta}{2}\right)},\qquad \mu_2=\tan\left(\frac{f(x_t)+f(\eta)}{2}\right)\tanh\left(\frac{x_t-\eta}{2}\right).
$$
We compute
$$
2\pax^2\arctan\left(\mu_1\right)=\frac{\tanh^2((x_t-\eta)/2)}{\cosh^2\left((x_t-\eta)/2\right)\cos^2(\theta)}\frac{Q_1(x_t,\eta,t)}{(\tanh^2((x_t-\eta)/2)+\tan^2(\theta))^2},
$$
with
$$
Q_1=\pax f(x_t)\mu_1^2+(1-(\pax f(x_t))^2)\mu_1-\pax f(x_t),
$$
and
$$
2\pax^2\arctan\left(\mu_2\right)=\frac{1}{\cosh^2\left((x_t-\eta)/2\right)\cos^2(\bar{\theta})}\frac{Q_2(x_t,\eta,t)}{(1+\tanh^2((x_t-\eta)/2)\tan^2(\bar{\theta}))^2},
$$
with
$$
Q_2=-\pax f(x_t)\mu_2^2+((\pax f(x_t))^2-1)\mu_2+\pax f(x_t).
$$
Thus the sign of the integral terms are given by the sign of $Q_1$ and $Q_2$. $Q_i$ are polynomials in the variables $\mu_i$,  respectively. 

The roots of $Q_1$ are $\pax f(x_t)$ and $-1/\pax f(x_t)$, so if we have
$$
\left|\mu_1\right|\leq \min\left\{\|\pax f(t)\|_{L^\infty},\frac{1}{\|\pax f(t)\|_{L^\infty}}\right\},
$$
then we can ensure that the integral involving the increments of $f$ is negative. However we have that for $\eta=x_t$ the following equality holds
$$
\lim_{\eta\rightarrow x_t}\left|\frac{\tan\left(\frac{f(x_t)-f(\eta)}{2}\right)}{\tanh\left(\frac{x_t-\eta}{2}\right)}\right|=\|\pax f(t)\|_{L^\infty},
$$
and so we need that 
$\min\left\{\|\pax f(t)\|_{L^\infty},1/\|\pax f(t)\|_{L^\infty}\right\}=\|\pax f(t)\|_{L^\infty}$. Thus we impose condition \eqref{H3}.

Moreover, if $|x-\eta|\geq1$ then 
$$
\left|\mu_1\right|\leq\frac{\tan\left(\|f_0\|_{L^\infty}\right)}{\tanh\left(\frac{1}{2}\right)}< \|\pax f_0\|_{L^\infty},
$$
under the hypothesis \eqref{H4}. Then 
$$
I_1^{out}=2\text{P.V.}\int_{B^c(x_t,1)}\pax^2\left(\arctan\left(\frac{\tan\left(\frac{f(x_t)-f(\eta)}{2}\right)}{\tanh\left(\frac{x_t-\eta}{2}\right)}\right)\right)d\eta<0.
$$
We have to bound the following integral
$$
I_1^{in}=\text{P.V.}\int_{B(x_t,1)}\frac{\pax f(x_t)\left(\mu_1\right)^2+(1-(\pax f(x_t))^2)\mu_1-\pax f(x_t)}{\sinh^2\left((x_t-\eta)/2\right)\cos^2(\theta)\left(1+\left(\mu_1 \right)^2\right)}d\eta.
$$
Using the definition of $\theta$ in \eqref{theta} and the fact that $|x_t-\eta|\leq 1$, we have 
$$
\frac{\tan(\theta)}{\tanh((x_t-\eta)/2)}-\pax f(x_t)\leq \frac{(x_t-\eta)^2}{48\tanh\left(\frac{1}{2}\right)}\left(\pax f(x_t)+5(\pax f(x_t))^3\right).
$$ 
To obtain this we split as follows
\begin{multline*}
\frac{\tan(\theta)}{\tanh((x_t-\eta)/2)}-\pax f(x_t)=\frac{\tan(\theta)-\theta}{\tanh((x_t-\eta)/2)}\\
+\theta\left(\frac{1}{\tanh((x_t-\eta)/2)}-\frac{2}{x_t-\eta}\right)+\frac{2\theta}{x_t-\eta}-\pax f(x_t).
\end{multline*}
Taylor theorem and the fact that the function $(x_t-\eta)/\tanh(x_t-\eta)\leq 0.5/\tanh(0.5)$ in this region give us the desired bound. In the same way,
\begin{multline*}
\frac{\tan^2(\theta)}{\tanh^2((x_t-\eta)/2)}-(\pax f(x_t))^2\\
\leq\frac{(x_t-\eta)^2}{48\tanh\left(\frac{1}{2}\right)}\left(\pax f(x_t)+5(\pax f(x_t))^3\right)\left(\pax f(x_t)+\frac{\tan\left(\frac{\pax f(x_t)}{2}\right)}{\tanh\left(\frac{1}{2}\right)}\right).
\end{multline*}
Thus, using the cancellation when $\mu_1=\pax f(x_t)$, we obtain
\begin{multline*}
I_1^{in}\leq\frac{\left(\pax f(x_t)+5(\pax f(x_t))^3\right)\left(1+\pax f(x_t)\left(\pax f(x_t)+\frac{\tan\left(\frac{\pax f(x_t)}{2}\right)}{\tanh\left(\frac{1}{2}\right)}\right)\right)}{24\tanh(1/2)\cos^2(\|f(t)\|_{L^\infty})}\int_{0}^1\frac{\eta^2d\eta}{\sinh^2\left(\frac{\eta}{2}\right)}\\
\leq \frac{\left(\pax f(x_t)+5(\pax f(x_t))^3\right)\left(1+\pax f(x_t)\left(\pax f(x_t)+\frac{\tan\left(\frac{\pax f(x_t)}{2}\right)}{\tanh\left(\frac{1}{2}\right)}\right)\right)}{6\tanh(1/2)\cos^2(\|f(t)\|_{L^\infty})}.
\end{multline*}
We have
$$
I_2\leq\left|\int_\RR\frac{((\pax f(x_t))^2-1)\tanh((x_t-\eta)/2)\tan(\bar{\theta})+\pax f(x_t)}{\cosh^2\left((x_t-\eta)/2\right)\cos^2(\bar{\theta})(1+\tanh^2((x_t-\eta)/2)\tan^2(\bar{\theta}))^2}d\eta\right|.
$$
Easily we get
$$
I_2\leq4\frac{\tan(\|f(t)\|_{L^\infty})+\|\pax f(t)\|_{L^\infty}}{\cos^2(\|f(t)\|_{L^\infty})}.
$$
It remains to show that
$$
I_1^{in}+I_2-8\pax f(x_t)\leq0.
$$
We need to use the local term $-8\pax f(x_t)$ in order to control the remainder terms. Using the maximum principle (Theorem \ref{teo2}), we obtain
\begin{multline*}
\frac{\|\pax f_0\|_{L^\infty}+5\|\pax f_0\|_{L^\infty}^3}{6\tanh(1/2)}
\left(1+\|\pax f_0\|_{L^\infty}\left(\|\pax f_0\|_{L^\infty}+\frac{\tan\left(\frac{\|\pax f_0\|_{L^\infty}}{2}\right)}{\tanh(1/2)}\right)\right)\\
+4\tan(\|f_0\|_{L^\infty})+4\|\pax f_0\|_{L^\infty}(1-2\cos^2(\|f_0\|_{L^\infty}))\leq0,
\end{multline*}
which is the condition \eqref{H5}. 

We have shown that, if initially the previous conditions hold, there is local in time decay for $\|\pax f(t)\|_{L^\infty}$, but maybe these conditions are not satisfied for all time. Indeed, if $\|\pax f(t)\|_{L^\infty}$ decays faster enough then the condition \eqref{H5} is not global. But, if there exists $t_1>0$ such that $\|\pax f\|_{L^\infty}$ starts to grow, then there is a time $t^*>t_1$ so that the condition \eqref{H5} again holds, and \eqref{MPD} is achieved. The same is valid for the condition \eqref{H4}.

The last part in the Theorem is obvious using the maximum principle for $\|f(t)\|_{L^\infty}$ (see Theorem \ref{teo2}). This concludes the proof. \end{proof}

\begin{coment}
We observe that in the case $\mathcal{A}=0$ the condition depends only on $\|\pax f_0\|_{L^\infty}$. In the case $0<\mathcal{A}<1$, this appears to be impossible because of two facts: First, the term with $\bar{\theta}$ in \eqref{theta} gives us a condition on $\|f_0\|_{L^\infty}$, also we notice that the condition on $\tan\left(\frac{f(x_t)-f(\eta)}{2}\right)/\tanh\left(\frac{x_t-\eta}{2}\right)$ gives us implicitly a condition on $\|f(t)\|_{L^\infty}$. Indeed, we have
$$
\left|\lim_{|\eta|\rightarrow\infty}\frac{\tan\left(\frac{f(x_t)-f(\eta)}{2}\right)}{\tanh\left(\frac{x_t-\eta}{2}\right)}\right|=|\tan(f(x_t)/2)|\leq\|\pax f (t)\|_{L^\infty}.
$$
The second fact is that the term $\mu_1$ can be bounded \emph{below} by the incremential quotients, but if we want to bound it \emph{above} we have to use $\|\pax f(t)\|_{L^\infty}$ and $\|f(t)\|_{L^\infty}$.
\end{coment}

\begin{coment}
The region without decay but with an uniform bound (see Figure~\ref{2d_0}) appears due to the boundaries. This region does not appear in the case with infinite depth $\mathcal{A}=0$. We notice that if we now take the $l\rightarrow\infty$ limit we recover the well-known result (contained in \cite{c-g09}) for the whole plane case. Indeed, if $l\rightarrow\infty$, the conditions \eqref{H4} and \eqref{H5} are automatically achieved and we only have \eqref{H3} as in \cite{c-g09}.
\end{coment}

\subsection{Turning waves}\label{sec3.2}

Now, we can prove the existence of turning waves in the stable Rayleigh-Taylor regime. 

\begin{thm}\label{sing}
Take $\rho^ 2-\rho^ 1>0$. Then there exist analytic initial data $s_0=s(\alpha,0)$, that can be parametrized as a graph, such that the solution of \eqref{eqcurva} at finite time is no longer a graph.
\end{thm}

\begin{proof} First, we show that there exists curves $z(\alpha)=(z_1(\alpha),z_2(\alpha))$ such that:
\begin{enumerate}
 \item[C1.] $z_i$ are analytic, odd functions.
 \item[C2.] $\partial_\alpha z_1(\alpha)>0, \forall \alpha\neq 0$, $\partial_\alpha z_1(0)=0$, and $\partial_\alpha z_2(0)>0$.
 \item[C3.] $\partial_\alpha v_1(0)=\partial_\alpha\pat z_1(0)<0$. 
\end{enumerate}
By integration by parts in expression \eqref{eqcurva} and using the definition of $z_i$ we obtain
\begin{multline}\label{eqvelsing}
\partial_\alpha v_1(0)=2\partial z_2(0)\int_0^\infty \partial_\alpha z_1(\eta)\sinh(z_1(\eta))\sin(z_2(\eta))\bigg{(}\frac{1}{(\cosh(z_1(\eta))-\cos(z_2(\eta)))^2}\\
+\frac{1}{(\cosh(z_1(\eta))+\cos(z_2(\eta)))^2}\bigg{)}d\eta.
\end{multline}
Now, we define piecewise smooth and odd curves $z(\alpha)=(z_1(\alpha),z_2(\alpha))$ (see Figure \ref{curvassing}) with components
$$
z_1(\alpha)=\alpha-\alpha\exp\left(-\alpha^2\right)
$$
and, fixed $2<b\leq a$ positive constants,
\begin{equation}\label{defz2}
z_2(\alpha)=\left\{
\begin{array}{lllll}\displaystyle \frac 1 a\sin(a\alpha) & \hbox{ if }  \displaystyle 0\leq \alpha\leq \frac \pi a,\\
\displaystyle \sin\left(\pi\frac{\alpha-(\pi/a)}{(\pi/a)-(\pi/b)}\right) & \text { if } \displaystyle\frac \pi a<\alpha<\frac\pi b,\\
\displaystyle -\alpha +\frac \pi b & \text{ if }\displaystyle  \frac \pi b\leq\alpha<\frac\pi 2,\\
\displaystyle\alpha - \pi+\frac\pi b & \text{ if } 
\displaystyle\frac \pi 2\leq\alpha<\pi (1-\frac 1 b),\\
\displaystyle 0 & \text{ if }
\displaystyle \pi(1-\frac 1 b)\leq\alpha.\\
                   \end{array}\right.
\end{equation}
Notice that C2 is achieved for this $(z_1,z_2)$.  Moreover, these curves satisfy the arc-chord condition in the whole domain. Using the definition of $z_2$ we have that 
$$
\partial_\alpha v_1(0)\leq2\partial z_2(0)(I_a+I_b),
$$
where $I_a,I_b$ are the integrals \eqref{eqvelsing} on the intervals $(0,\pi/a)$ and $(\pi/b,\pi)$, respectively. Easily, we show $I_b<0$ and this is independent of the choice of $a$. The integral $I_a$ is well defined and positive, but goes to zero as $a$ grows. Therefore, by approximating, there exists curves $(z_1,z_2)$ that satisfies the conditions C1--C3.

Now, we consider $(z_1,z_2)$ as the analytic initial datum for the equation \eqref{eqcurva}. By a Cauchy-Kowalevski Theorem, there exists a curve, $w(\alpha,t)$, solution of \eqref{eqcurva} for any $t\in[-T,T]$ (see Section \ref{sec2.3}). Due to C3, we get the following
\begin{enumerate}
\item for $-T<t<0$, we have $\min_\alpha \paa w_1(\alpha,t)>\min_\alpha \paa w_1(\alpha,0)=0$ and $s_0(\alpha)=s(\alpha,0)=w(\alpha,-T/2)$ can be parametrized as a graph.\\
\item At $t=0$, $w(\alpha,0)=s(\alpha,T/2)=z(\alpha)$ has a vertical tangent.\\
\item For $0<t<T$ we get $\min_\alpha \paa w_1(\alpha,t)<\min_\alpha \paa w_1(\alpha,0)=0$. Thus, for $0<t<T$, the curve is no longer a graph and the Rayleigh-Taylor condition is not satisfied in a neighbourhood of $\alpha=0$.
\end{enumerate}
\end{proof}

\begin{coment} This theorem implies that there exist initial data $f_0$, parametrized as graphs, such that the solution of \eqref{eq0.1} develops a blow up for $\|\pax f(t)\|_{L^\infty}$ at finite time $t_1$.
\end{coment}

\begin{figure}[t]
		\begin{center}
		\includegraphics[scale=0.35]{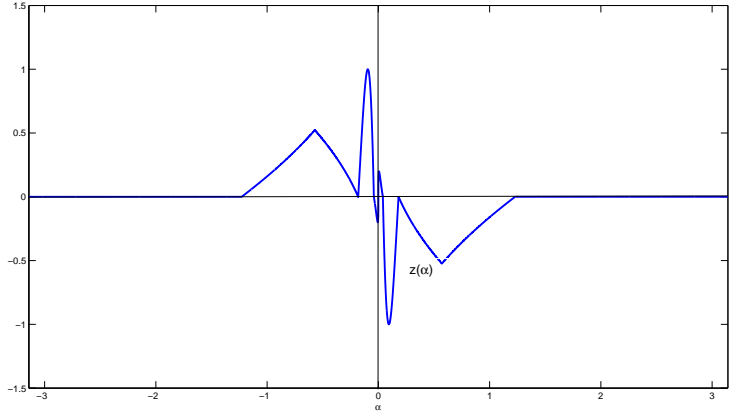} 
		\end{center}
		\caption{The curve in the case $a=5,b=3$.}
\label{curvassing}
\end{figure}

\subsection{Numerical evidence}\label{sec3.3}
In this Section we obtain firm numerical evidence showing that the confined problem is more singular than the problem with infinite depth \eqref{full}. The precise statement of this fact is the following: {\it We consider a strip with width equal to $l$, a fixed constant. Then there exists initial data $z_0(\alpha)=(z_1(\alpha),z_2(\alpha))$ that can not be parametrized as graphs such that the solution of $\eqref{eqcurva}$ achieve the (Rayleigh-Taylor) unstable case and, if you consider the same initial datum when the depth is infinite, the same curves becomes graphs.}

It is enough to show that there exist smooth curves $z(\alpha,0)=(z_1(\alpha,0),z_2(\alpha,0))$ satisfying arc-chord condition and such that $\paa z_1(0,0)=0$ and the following holds:
\begin{enumerate}
 \item $\partial_\alpha v_1(0,0)=\partial_\alpha\pat z_1(0,0)>0$ in the deep water regime,
 \item $\partial_\alpha v_1(0,0)=\partial_\alpha\pat z_1(0,0)<0$  when the strip is considered.
\end{enumerate}
Indeed, if $\paa v_1(0,0)=\partial_\alpha\pat z_1(0,0)>0$ then denoting $m(t)=\min_\alpha \paa z_1(\alpha,t)$, we have $m(0)=\paa z_1(0,0)=0$ and $\frac{d}{dt}m(t)>0$ for $t>0$ small enough. This implies $m(\delta)>0$ for a small enough $\delta>0$ and the curve can be parametrized as a graph. If $\paa v_1(0,0)=\partial_\alpha\pat z_1(0,0)<0$, then $m(t)<0$ if $t$ is small enough and the curve can not be parametrized as a graph.

We construct a piecewise smooth curve such that both conditions holds (see Figure \ref{curvassing}). We take $z_1$ defined as follows
$$
z_1(\alpha)=\alpha-e^{-\alpha^2 k}\sin(\alpha),
$$
with $k=10^{-4}$. The idea is to take $k<<1$ such that $e^{-\alpha^2 k}\approx 1,$ for $-\pi<\alpha<\pi$. Moreover, we take $z_2$ as in \eqref{defz2} with $a=b=3$, i.e.,
\begin{equation*}
z_2(\alpha)=\left\{
\begin{array}{lllll}\displaystyle \frac 1 3\sin(3\alpha) & \hbox{ if }  \displaystyle 0\leq \alpha\leq \frac \pi 3,\\
\displaystyle -\alpha +\frac \pi 3 & \text{ if }\displaystyle  \frac \pi 3\leq\alpha<\frac\pi 2,\\
\displaystyle\alpha - \frac{2\pi} 3 & \text{ if } 
\displaystyle\frac \pi 2\leq\alpha<\frac{2\pi} 3,\\
\displaystyle 0 & \text{ if }
\displaystyle \frac{2\pi} 3\leq\alpha.\\
                   \end{array}\right.
\end{equation*} 
Notice that, in the deep water regime, the expression \eqref{eqvelsing} takes the form
$$
\frac{\partial_\alpha v_1(0)}{2}=4\partial z_2(0)\int_0^\infty \frac{\partial_\alpha z_1(\eta)z_1(\eta)z_2(\eta)}{(z_1(\eta))^2+(z_2(\eta))^2)^2}d\eta.
$$
Substituting the choice of $z$, we need to compute
\begin{equation}\label{comsing1}
\frac{\partial_\alpha v_1(0)}{2}=\mathcal{I}_1+\mathcal{I}_2+\mathcal{I}_3,
\end{equation}
where
$$
\mathcal{I}_1=\frac{4}{3}\int_0^\frac{\pi}{3} \frac{(1-\cos(\eta)e^{-\eta^2 k}+2k\eta e^{-\eta^2 k}\sin(\eta))(\eta-e^{-\eta^2 k}\sin(\eta))\sin(3\eta)
}{(\eta-e^{-\eta^2 k}\sin(\eta))^2+(\sin(3\eta)/3
)^2)^2}d\eta,
$$
$$
\mathcal{I}_2=4\int_\frac{\pi}{3}^\frac{\pi}{2} \frac{(1-\cos(\eta)e^{-\eta^2 k}+2k\eta e^{-\eta^2 k}\sin(\eta))(\eta-e^{-\eta^2 k}\sin(\eta))(-\eta +\pi/3)}{(\eta-e^{-\eta^2 k}\sin(\eta))^2+(-\eta +\pi/3)^2)^2}d\eta,
$$
and
$$
\mathcal{I}_3=4\int_\frac{\pi}{2}^\frac{2\pi}{3} \frac{(1-\cos(\eta)e^{-\eta^2 k}+2k\eta e^{-\eta^2 k}\sin(\eta))(\eta-e^{-\eta^2 k}\sin(\eta))(\eta -2\pi/3)}{(\eta-e^{-\eta^2 k}\sin(\eta))^2+(\eta-2\pi/3)^2)^2}d\eta.
$$
In the finite depth case the integrals appearing in \eqref{eqvelsing} are
\begin{equation}\label{comsing2}
\frac{\partial_\alpha v_1(0)}{2}=\mathcal{I}_4+\mathcal{I}_5+\mathcal{I}_6,
\end{equation}
where
\begin{multline*}
\mathcal{I}_4=\int_0^\frac{\pi}{3} (1-\cos(\eta)e^{-\eta^2 k}+2k\eta e^{-\eta^2 k}\sin(\eta))\sinh(\eta-e^{-\eta^2 k}\sin(\eta))\sin(\sin(3\eta)/3
)\\\cdot\bigg{(}\frac{1}{(\cosh(\eta-e^{-\eta^2 k}\sin(\eta))-\cos(\sin(3\eta)/3
))^2}\\
+\frac{1}{(\cosh(\eta-e^{-\eta^2 k}\sin(\eta))+\cos(\sin(3\eta)/3
))^2}\bigg{)}d\eta,
\end{multline*}

\begin{multline*}
\mathcal{I}_5=\int_\frac{\pi}{3}^\frac{\pi}{2} (1-\cos(\eta)e^{-\eta^2 k}+2k\eta e^{-\eta^2 k}\sin(\eta))\sinh(\eta-e^{-\eta^2 k}\sin(\eta))\sin(-\eta +\pi/3)\\
\cdot\bigg{(}\frac{1}{(\cosh(\eta-e^{-\eta^2 k}\sin(\eta))-\cos(-\eta +\pi/3))^2}\\
+\frac{1}{(\cosh(\eta-e^{-\eta^2 k}\sin(\eta))+\cos(-\eta +\pi/3))^2}\bigg{)}d\eta,
\end{multline*}
and
\begin{multline*}
\mathcal{I}_6=\int_\frac{\pi}{2}^\frac{2\pi}{3} (1-\cos(\eta)e^{-\eta^2 k}+2k\eta e^{-\eta^2 k}\sin(\eta))\sinh(\eta-e^{-\eta^2 k}\sin(\eta))\sin(\eta-2\pi/3)\\
\cdot\bigg{(}\frac{1}{(\cosh(\eta-e^{-\eta^2 k}\sin(\eta))-\cos(\eta-2\pi/3))^2}\\
+\frac{1}{(\cosh(\eta-e^{-\eta^2 k}\sin(\eta))+\cos(\eta-2\pi/3))^2}\bigg{)}d\eta.
\end{multline*}

In order to obtain the sign of \eqref{comsing1} and \eqref{comsing2}, we compute the integrals $\mathcal{I}_i,\; (i=2,3,5,6)$ using the trapezoidal rule with a fine enough mesh (see Figure \ref{sing1}). The integrals $\mathcal{I}_i,\;(i=1,4)$ are approximated by
$$
\mathcal{I}'_1=\frac{4}{3}\int_{0.1}^\frac{\pi}{3} \frac{(1-\cos(\eta)e^{-\eta^2 k}+2k\eta e^{-\eta^2 k}\sin(\eta))(\eta-e^{-\eta^2 k}\sin(\eta))\sin(3\eta)
}{(\eta-e^{-\eta^2 k}\sin(\eta))^2+(\sin(3\eta)/3
)^2)^2}d\eta,
$$
and
\begin{multline*}
\mathcal{I}'_4=\int_{0.1}^\frac{\pi}{3} (1-\cos(\eta)e^{-\eta^2 k}+2k\eta e^{-\eta^2 k}\sin(\eta))\sinh(\eta-e^{-\eta^2 k}\sin(\eta))\sin(\sin(3\eta)/3
)\\\cdot\bigg{(}\frac{1}{(\cosh(\eta-e^{-\eta^2 k}\sin(\eta))-\cos(\sin(3\eta)/3
))^2}\\
+\frac{1}{(\cosh(\eta-e^{-\eta^2 k}\sin(\eta))+\cos(\sin(3\eta)/3
))^2}\bigg{)}d\eta.
\end{multline*}
The truncation of the integral domains in $\mathcal{I}'_i,\; (i=1,4)$ gives us an error $E_{PV}\leq 0.72\cdot 10^{-3}$. To obtain this bound we notice that, due to the particular choice of $z_i$,
$$
\int_0^{x}\frac{\partial_\alpha z_1(\eta)z_1(\eta)z_2(\eta)}{(z_1(\eta))^2+(z_2(\eta))^2)^2}d\eta= O(x^3),
$$
and the same is valid for the relevant integral in the presence of boundaries \eqref{eqvelsing}.
 
The other error is coming from the method used in the numerical quadrature. We use the trapezoidal rule, obtaining $E_I\leq 1.1\cdot10^{-3}$. We conclude that, if $\hat{\paa v_1(0)}$ denotes the numerical approximation of $\paa v_1(0)$ defined in \eqref{comsing2}, we have
$$
\paa v_1(0)\leq \hat{\paa v_1(0)}+|E_{PV}|+|E_{I}|<0,
$$ 
and, analogously, in the case where $\paa v_1(0)$ is defined in \eqref{comsing1} we get
$$
0< \hat{\paa v_1(0)}-|E_{PV}|-|E_{I}|\leq\paa v_1(0).
$$

Finally, we approximate this $z_0$ by analytic functions. This shows that the problem with finite depth appears to be, in this precise sense, more singular than the case $\mathcal{A}=0$.

In order to complete a rigorous enclosure of the integral, we are left with the bounding of the errors coming from the floating point representation and the computer operations and their propagation. In a forthcoming paper (see \cite{GG}) we will deal with this matter. By using interval arithmetics, we will give a computer assisted proof of this result.

\begin{figure}[t]
		\begin{center}
		\includegraphics[scale=0.3]{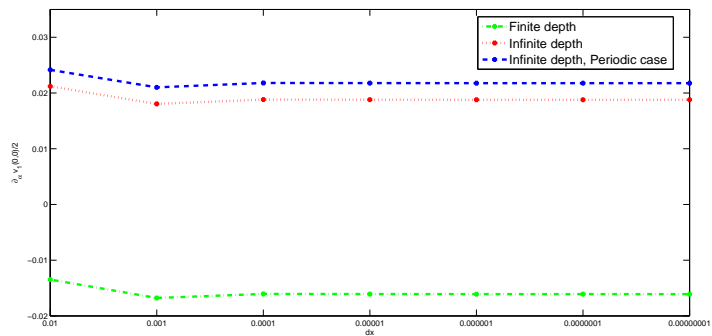} 
		\end{center}
		\caption{Approximating $\partial_\alpha v_1(0)/2$ with different spatial step $dx$.}		
		\label{sing1}
\end{figure}

\medskip

{\bf Acknowledgement.} The authors are supported by the Grants MTM2011-26696 and SEV-2011-0087 from Ministerio de Ciencia e Innovaci\'on (MICINN). Diego C\' ordoba was partially supported by StG-203138CDSIF of the ERC. The authors are grateful to A. Castro and F. Gancedo for their helpful comments during the preparation of this work. The authors would like to thank the
referees for their help in improving the manuscript.


\medskip

\begin{thebibliography}{10}
 
\bibitem{ambrose2004well}
D.~Ambrose.
\newblock Well-posedness of two-phase Hele-Shaw flow without surface tension.
\newblock {\em European Journal of Applied Mathematics}, 15(5):597--607, 2004.

\bibitem{bakan2007hardy}
A.~Bakan and S.~Kaijser.
\newblock Hardy spaces for the strip.
\newblock {\em Journal of mathematical analysis and applications},
  333(1):347--364, 2007.

\bibitem{bear}
J.~Bear.
\newblock {\em {Dynamics of fluids in porous media}}.
\newblock Dover Publications, 1988.

\bibitem{bona2008asymptotic}
J.~Bona, D.~Lannes, and J.~Saut.
\newblock Asymptotic models for internal waves.
\newblock {\em Journal de Math{\'e}matiques Pures et Appliqu{\'e}s},
  89(6):538--566, 2008.

\bibitem{castro2012breakdown}
A.~Castro, D.~Cordoba, C.~Fefferman, and F.~Gancedo.
\newblock Breakdown of smoothness for the Muskat problem.
\newblock {\em To appear in Arch. Rat. Mech. Anal.}, 2012.

\bibitem{ccfgl}
A.~Castro, D.~Cordoba, C.~Fefferman, F.~Gancedo, and M.~Lopez-Fernandez.
\newblock Rayleigh-Taylor breakdown for the Muskat problem with applications to
  water waves.
\newblock {\em Annals of Math}, 175:909--948, 2012.

\bibitem{ccgs-10}
P.~Constantin, D.~Cordoba, F.~Gancedo, and R.~Strain.
\newblock On the global existence for the Muskat problem.
\newblock {\em J. Eur. Math. Soc.}, 15, 201-227, 2013.

\bibitem{c-c-g10}
A.~Cordoba, D.~C{\'o}rdoba, and F.~Gancedo.
\newblock {Interface evolution: the Hele-Shaw and Muskat problems}.
\newblock {\em Annals of Math}, 173, no. 1:477--542, 2011.

\bibitem{c-g07}
D.~C{\'o}rdoba and F.~Gancedo.
\newblock {Contour dynamics of incompressible 3-D fluids in a porous medium
  with different densities}.
\newblock {\em Communications in Mathematical Physics}, 273(2):445--471, 2007.

\bibitem{c-g09}
D.~C{\'o}rdoba and F.~Gancedo.
\newblock A maximum principle for the Muskat problem for fluids with different
  densities.
\newblock {\em Communications in Mathematical Physics}, 286(2):681--696, 2009.

\bibitem{c-g-o08}
D.~C{\'o}rdoba, F.~Gancedo, and R.~Orive.
\newblock A note on interface dynamics for convection in porous media.
\newblock {\em Physica D: Nonlinear Phenomena}, 237(10-12):1488--1497, 2008.

\bibitem{e-m10}
J.~Escher and B.~Matioc.
\newblock On the parabolicity of the Muskat problem: Well-posedness, fingering,
  and stability results.
\newblock {\em Arxiv preprint arXiv:1005.2512}, 2010.

\bibitem{F}
A.~Friedman.
\newblock Free boundary problems arising in tumor models.
\newblock {\em Atti Accad. Naz. Lincei Cl. Sci. Fis. Mat. Natur. Rend.
  Lincei,}, 9(3-4), 2004.

\bibitem{GG}
 J. G\'omez-Serrano and R.Granero-Belinch\'on.
 \newblock On turning waves for the inhomogeneous Muskat problem: a computer-assisted proof,
 \newblock {\em Preprint}.

\bibitem{H-S}
H.~Hele-Shaw.
\newblock Flow of water.
\newblock {\em Nature}, 58(1509):520--520, 1898.

\bibitem{KK}
H.~Kawarada and H.~Koshigoe.
\newblock Unsteady flow in porous media with a free surface.
\newblock {\em Japan Journal of Industrial and Applied Mathematics},
  8(1):41--84, 1991.

\bibitem{knupfer2010darcy}
H.~Kn{\"u}pfer and N.~Masmoudi.
\newblock Darcy flow on a plate with prescribed contact angle—well-posedness
  and lubrication approximation.
\newblock {\em Preprint}, 2010.

\bibitem{bertozzi-Majda}
A.~Majda and A.~Bertozzi.
\newblock {\em {Vorticity and incompressible flow}}.
\newblock Cambridge Univ Pr, 2002.

\bibitem{Muskat}
M.~Muskat.
\newblock The flow of homogeneous fluids through porous media.
\newblock {\em Soil Science}, 46(2):169, 1938.

\bibitem{bn}
D.~Nield and A.~Bejan.
\newblock {\em {Convection in porous media}}.
\newblock Springer Verlag, 2006.

\bibitem{nirenberg1972abstract}
L.~Nirenberg.
\newblock An abstract form of the nonlinear Cauchy-Kowalewski theorem.
\newblock {\em J. Differential Geometry}, 6:561--576, 1972.

\bibitem{nishida1977note}
T.~Nishida.
\newblock A note on a theorem of Nirenberg.
\newblock {\em J. Differential Geometry}, 12:629--633, 1977.

\bibitem{P}
C.~Pozrikidis.
\newblock Numerical simulation of blood and interstitial flow through a solid
  tumor.
\newblock {\em Journal of Mathematical Biology}, 60(1):75--94, 2010.

\bibitem{SCH}
M.~Siegel, R.~Caflisch, and S.~Howison.
\newblock Global existence, singular solutions, and ill-posedness for the
  Muskat problem.
\newblock {\em Communications on Pure and Applied Mathematics},
  57(10):1374--1411, 2004.
 
\end{thebibliography}
\end{document}